\documentclass[12pt]{article}

\usepackage{pmat}
\usepackage{cite}
\usepackage{amsmath}
\usepackage{amssymb}
\usepackage{arydshln}

\usepackage{amsthm}
\usepackage{multirow}
\usepackage{multicol}
\usepackage{color}
\usepackage{graphicx}         
\usepackage{subfigure}
\usepackage{epstopdf}                           
\def\disp{\displaystyle}
\def\dref#1{(\ref{#1})}
\def\dfrac{\displaystyle\frac}
\newtheorem{theorem}{Theorem}                   
                  
\newtheorem{proposition}{Proposition}
\newtheorem{lemma}{Lemma}

\newtheorem{remark}{Remark}

\makeatletter
\newcommand{\Rmnum}[1]{\expandafter\@slowromancap\romannumeral #1@}
\makeatother

\begin{document}

		\title{ Sampled-data control of 2D  Kuramoto-Sivashinsky equation}
		\author{Wen Kang and Emilia Fridman
		 \thanks{W. Kang (email: kangwen@amss.ac.cn) is with School of Automation and Electrical Engineering, University of Science and Technology Beijing, P. R. China,  and also with School of Electrical Engineering, Tel Aviv University, Israel.}
\thanks{E. Fridman (email:   emilia@eng.tau.ac.il) is with School of Electrical Engineering, Tel Aviv University, Israel.}
	\thanks{ This work was supported by Israel Science Foundation (grant No. 673/19) and National Natural Science Foundation of China (Grant No. 61803026).}}
	\maketitle	

         \begin{abstract}
This paper addresses sampled-data control of 2D Kuramoto-Sivashinsky equation over a rectangular domain $\Omega$. We suggest
to divide the 2D rectangular $\Omega$ into $N$  sub-domains, where sensors provide spatially averaged or point state measurements
to be transmitted through communication network to the controller. Note that differently from 2D heat equation, here we manage with sampled-data control under point measurements. We design a regionally stabilizing controller applied through distributed in space characteristic functions. Sufficient conditions ensuring regional stability of the closed-loop system are established in terms of linear matrix inequalities (LMIs). By solving these LMIs, we find an estimate on the set of initial conditions starting from which the state trajectories of the system are exponentially converging to zero. A numerical example demonstrates the efficiency of the results.
         \end{abstract}
      

	\section{Introduction}

In recent decades, Kuramoto-Sivashinsky equation (KSE)
 has drawn a lot of attention as a nonlinear model
of pattern formations on unstable flame fronts and thin hydrodynamic films (see e.g. \cite{kuramoto,bao}). KSE arises in the study of thin liquid films, exhibiting a wide range of dynamics in different parameter regimes, including unbounded growth and full spatiotemporal chaos. For 1D
KSE,  distributed control  (see e.g. \cite{Armaou2,Armaou1,Armaou,Titi}) has been considered. Boundary stabilization of KSE has been studied in   \cite{Krstic,LU}.


 Sampled-data control of PDEs became recently an active research area (see e.g. \cite{kang,Emilia3,Emilia2,kang2}) for practical application of finite-dimensional controllers for PDEs,
where LMI conditions for the exponential/regional  stability  of the closed-loop systems were derived
in the framework of time-delay approach by employing appropriate Lyapunov functionals.
Most of the existing results deal with 1D PDEs system. Sampled-data observers for ND
and 2D heat equations with globally
 Lipschitz nonlinearities have been suggested in \cite{Emilia1} and \cite{Anton1}.   However, the
above results were confined to diffusion equations. 
Sampled-data control of various classes of high dimensional PDEs is an interesting and challenging problem. 

In our recent paper \cite{kang}, we have suggested sampled-data control of 1D KSE, where both point and averaged state measurements were studied.  In this paper we aim to extend results of \cite{kang} to 2D case.
Extension from 1D  (\cite{kang,ECC18}) to 2D  is far from being straightforward. Thus, in the case of heat equation, sampled data extension under the point measurements seems to be not possible (see \cite{Anton1}).This is due to the fact that stability analysis
of the closed-loop system is based on the bound of the $L^2$-norm of the difference between the state and its point value. However,
according to Friedrich's inequality  \cite{Anton1,jones}, this bound depends on the $L^2$-norm of the second-order spatial derivatives of the state.  Differently from the heat equation, stability analysis for KSE in $H^2$ allows to compensate such terms. 
We establish stability analysis of the closed-loop sampled-data system by constructing an appropriate Lyapunov-Krasovskii functional. Some preliminary results under averaged measurements were  presented in \cite{CDC19}, where the sampled-data case is limited to averaged measurements and there is no detailed proof of the well-posedness.

In the present paper, we  design a sampled-data controller for 2D KSE under averaged/point measurements based on LMIs. 
In comparison to the existing known results, new  special challenges  of this work are the following:\\
1) The present paper  gives the first  extension to 2D PDE in the case of sampled-data point measurements.
The results from \cite{Anton1} cannot be extended to the case of  point measurements. 
This is due to  the second order spatial derivative in Lemma \ref{anton} that cannot be compensated in  Lyapunov analysis. 
\\ 
 2) Here we have  the nonlinear term ``$zz_{x_1}$" which is locally Lipschitz in $D((-A)^{\frac{1}{2}})$ (defined in Section III below).  Due to the nonlinear term, the challenge is to find a bound on the domain of attraction. This bound is based on the new 2D Sobolev inequality (Lemma \ref{kka}) that we have derived.   Lemma \ref{kka} bounds  a function in the $C^0$-norm using  $L^2$-norms of its first and second spatial derivatives.  In \cite{Emilia1} and \cite{Anton1}, the nonlinear term is  subject to sector bound inequality which holds globally and leads to global results.\\
3)  The well-posedness is  challenging. We have provided  more detailed proof for this, and shown that $A$ generates an analytic semigroup even for rectangular domain $\Omega$ (non $C^1$ boundary). 


The remainder of this work is organized as follows. Useful lemmas are introduced and the problem setting  is reported  in Section II. Sections III-IV are devoted to construction of  continuous static output-feedback/sampled-data controllers  under the averaged or point measurements.
In Section V,  a
numerical example is carried out to illustrate the efficiency of the  main results.
 Finally, some
concluding remarks and possible future research lines are presented in Section VI.

{\bf Notation} The superscript $``T"$ stands for matrix transposition, $\mathbb R^n$ denotes
the $n$-dimensional Euclidean space with the norm $ | \cdot |$.  $\Omega\subset \mathbb R^2$ denotes a computational domain, $L^2(\Omega)$ denotes the space of measurable squared-integrable functions over $\Omega$ with the corresponding
 norm $\|z\|_{L^2(\Omega)}^2=\int_\Omega |z(x)|^2dx$.
Let $\partial \Omega$ be the boundary of $\Omega$. 
The Sobolev space $H^{k}(\Omega)$ is defined as
 $H^{k}(\Omega)=\{z:D^\alpha z\in L^2(\Omega),\; \forall\;0\le |\alpha|\le k\}$ with norm
 $\|z\|_{H^{k}(\Omega)=\left(\sum\limits_{0\le |\alpha|\le k}\|D^\alpha z\|^2_{L^2(\Omega)}\right)}^\frac{1}{2}.$
The space $H^k_0(\Omega)$ is the closure of $C_c^\infty(\Omega)$ in the space $H^k(\Omega)$ with the norm  $\|z\|_{H_0^{k}(\Omega)}=\left(\sum\limits_{ |\alpha|= k}\|D^\alpha z\|^2_{L^2(\Omega)}\right)^\frac{1}{2}.$

\section{Problem formulation and useful lemmas}
Denote by $\Omega$ the two dimensional (2D) unit square
 $$\Omega=[0,1]\times [0,1]\subset \mathbb{R}^2.$$
Consider the biharmonic operator:
\begin{equation*}
\Delta^2=\dfrac{\partial ^4}{\partial x_1^4}+2\dfrac{\partial ^4}{\partial x_1^2\partial x_2^2}+\dfrac{\partial ^4}{\partial x_2^4}.
\end{equation*}
 
As in \cite{Ruben}, we consider the following 2D  Kuramoto-Sivashinsky equation (KSE)  over $\Omega$ under the Dirichlet boundary conditions:
\begin{equation}\label{a}
	\left\{\begin{array}{ll}
\!\!	z_t+zz_{x_1}\!+ \!(1-\kappa) z_{x_1x_1}\!-\!\kappa z_{x_2x_2}\!+\! \Delta^2 z\!=\!\sum\limits_{j=1}^N\chi_{j}(x) U_{j}(t), \\\hspace{4.5cm}(x,t)\in \Omega\times [0,\infty),\\
	z|_{\partial \Omega}=0, \;\dfrac{ \partial z}{\partial n}|_{\partial \Omega}=0,\\
	z(x,0)=z_0(x),
	\end{array}\right.
	\end{equation}
where  $x=(x_1,x_2)\in \Omega$,   $z\in \mathbb R$ is the state of KSE,  $\dfrac{ \partial z}{\partial n}$ is the
 normal derivative,
 and $U_{j}(t)\in \mathbb{R}$, $j=1,2,\cdots, N$ are the control inputs. Here the parameter $\kappa$ denotes the angle of the substrate to the horizontal: \\
 for $\kappa> 0$ we have overlying film flows, a vertical film flow for $\kappa = 0$, and hanging flows when $\kappa< 0$.

Motivated by \cite{Titi,Emilia3,Emilia2,Emilia1,Titi2} 	we suggest to 
 divide $\Omega$ into $N$ square sub-domains $\Omega_{j}$ covering the whole region $\cup_{j=1}^{N} \Omega_j =\Omega$   (see Fig. 1) with an actuator and a sensor placed in each $\Omega_j$.
 Here 
 $$\begin{array}{ll}\Omega_j=\{x=(x_1,x_2)^T\in \Omega |x_i\in [x_i^{\rm min}(j), x_i^{\rm max}(j)], i=1,2\},\\
\hspace{6cm} j=1,2,\cdots, N.
 \end{array}$$
The measure of their intersections is zero.  
 \begin{figure}[h]
\begin{center}	
	{\includegraphics[width=7cm]{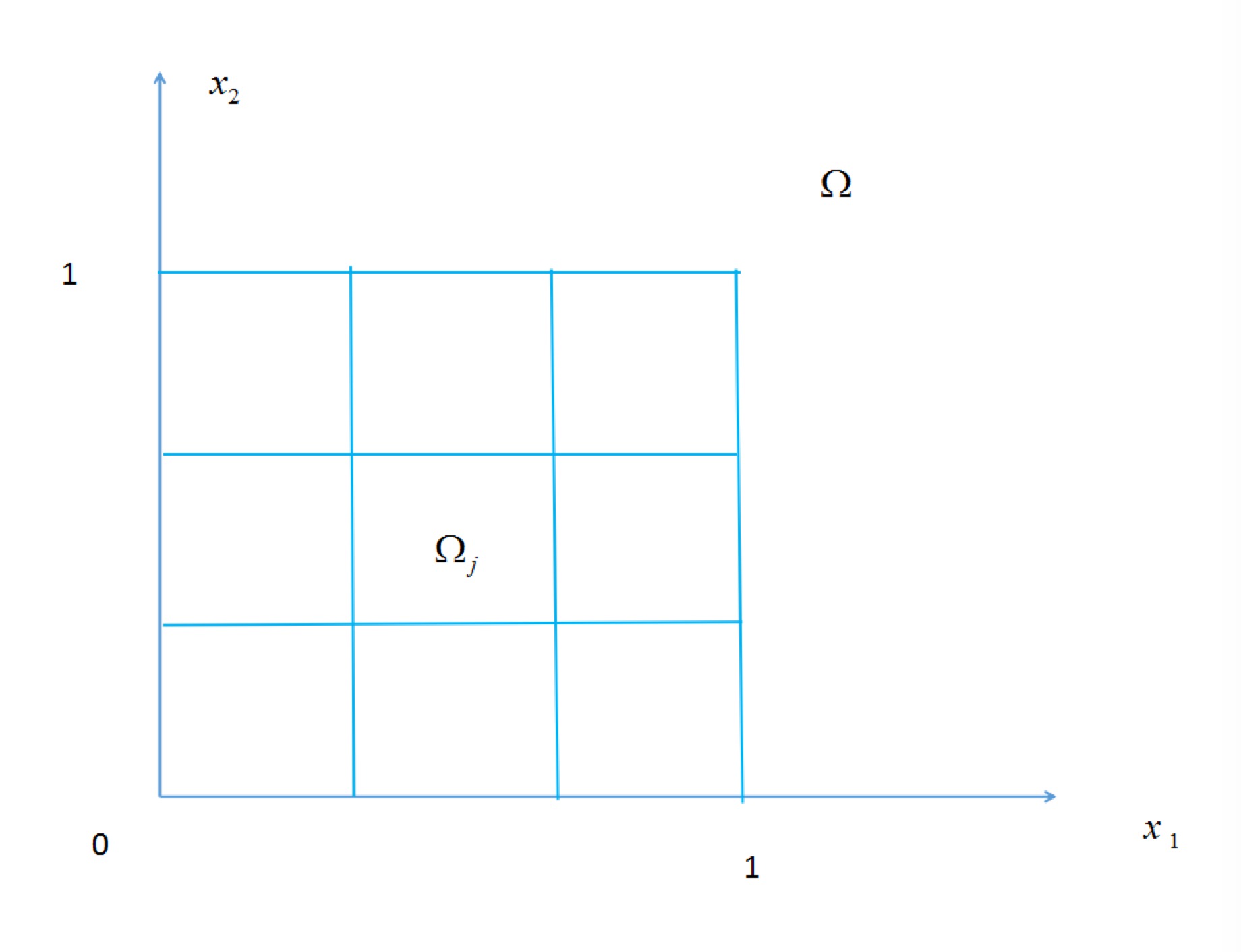}}
	\caption{Unit square $\Omega$ and  sub-domains $\Omega_j$}
	\end{center}
\end{figure}
	Let
$$0=t_0<t_1<\cdots<t_k\cdots, \quad \lim\limits_{k\to \infty}t_k=\infty $$
  be sampling time instants.
	The sampling sub-domains in time and in space  may be  bounded,
	$$
	\begin{array}{ll}
	0\le t_{k+1}-t_{k}\le h, \\
	0<x_i^{\rm max}(j)-x_i^{\rm min}(j)= \Delta_j\le  \bar \Delta, \;\\
\hspace{2cm} i=1,2;	\;\;j=1,2,\cdots, N,
\end{array}
$$
where $h$ and $\bar \Delta$ are the corresponding upper bounds.
 \begin{remark}
 For simplicity, we consider each sub-domain $\Omega_j$ is a square (i.e. $x_1^{\rm max}(j)-x_1^{\rm min}(j))=x_2^{\rm max}(j)-x_2^{\rm min}(j),\;j=1,2,\cdots, N)$. Indeed,  $\Omega_j$  can be a rectangular.  For the case that $x_1^{\rm max}(j)-x_1^{\rm min}(j))\neq x_2^{\rm max}(j)-x_2^{\rm min}(j)$ for some $j$ (see Fig. 1 of \cite{Anton1}),  $\Delta_j$ can be chosen as follows $$\Delta_j=\max\{x_1^{\rm max}(j)-x_1^{\rm min}(j),x_2^{\rm max}(j)-x_2^{\rm min}(j)\}.$$Thus, the results of this work are applicable to the case of nonsquare sub-domains. 
 \end{remark}
The spatial   characteristic functions are taken as
	\begin{equation}\label{qin}
	\left\{\begin{array}{ll}\chi_{j}(x)= 1,\; x\in \Omega_{j},\\
\chi_{j}(x)= 	0 ,\; x \notin \Omega_{j},
	 \end{array}\right. \;j=1,\cdots, N.
	\end{equation}
We assume that sensors provide   the following averaged measurements
	\begin{equation}\label{w}
	\begin{array}{ll}
	y_{jk}=\dfrac{\int_{\Omega_{j}}z(x,t_k)dx}{|\Omega_{j}|},\;
	\;j=1,\dots, N;\;k=0,1,2\dots
	\end{array}
	\end{equation}
	or point measurements
		\begin{equation}\label{swift}
	\begin{array}{ll}
	y_{jk}=z(\bar x_j,t_k),\;
	\;j=1,\dots, N;\;k=0,1,2\dots
	\end{array}
	\end{equation}	
       where  $\bar x_j$ locates in the center of the square sub-domain $\Omega_j$, and $|\Omega_{j}|$   stands for the Lebesgue measure of the domain $\Omega_{j}$.

	We aim to design for \dref{a} an exponentially stabilizing sampled-data controller  that can be implemented by zero-order hold devices:
	\begin{equation}\label{y}
	\begin{array}{ll}
	U_{j}(t)=-\mu y_{jk},\;\;j=1,\dots, N;\\
	\hspace{2cm}t\in [t_k,t_{k+1}),\;k=0,1,2\dots,
	\end{array}
	\end{equation}
	where $\mu$ is a positive controller gain 
and $y_{jk}$ is given by \dref{w} or \dref{swift}.

We present below some useful lemmas:
\begin{lemma}\label{d}
Let  $\Omega=(0,L_1)\times (0,L_2) $. Assume $f:\Omega\to \mathbb R$ and $f\in H^1(\Omega)$.\\
(i) (Poincar\'e's inequality) If $\int_\Omega f(x)dx=0$, then according to \cite{Payne}
\begin{equation}
\|f\|_{L^2(\Omega)}^2\le \dfrac{L_1^2+L_2^2}{\pi^2}\|\nabla f\|_{L^2(\Omega)}^2. 
\end{equation}
(ii) (Wirtinger's inequality) \cite{Emilia1}  If $f|_{\partial \Omega}= 0$, then the following inequality
holds:
\begin{equation}
\|f\|_{L^2(\Omega)}^2\le  \dfrac{L_1^2+L_2^2}{\pi^2}\|\nabla f\|_{L^2(\Omega)}^2.
\end{equation}
\end{lemma}
The following lemma gives a classical Friedrich's inequality (Theorem 18.1 of \cite{pingping}) with tight bounds on the coefficients  of terms $\left\|\dfrac{\partial f}{\partial x_1}\right\|^2_{L^2(\Omega)}$,  $\left\|\dfrac{\partial f}{\partial x_2}\right\|^2_{L^2(\Omega)}$ and $\left\|\dfrac{\partial^2 f}{\partial x_1\partial x_2}\right\|^2_{L^2(\Omega)}$. The inequality \dref{pinga} bounds the $L^2$-norm of  a function  by the reciprocally convex combination of the $L^2$-norm of its derivatives. 	
	\begin{lemma}\label{qinqin}(see (2) of \cite{Anton1})
	Let $\Omega=(0,l)^2 $,  $f\in H^2(\Omega)$ with $f(0,0)=0$. Then 
	the following inequality holds:
	\begin{equation}\label{pinga}
	\begin{array}{ll}
	\|f\|_{L^2(\Omega)}^2\!&\le \dfrac{1}{\alpha_1}\!\left(\dfrac{2l}{\pi}\right)^2\!\!\left\|\dfrac{\partial f}{\partial x_1}\right\|_{L^2(\Omega)}^2\! \!\! \!\!+\!\dfrac{1}{\alpha_2}\left(\dfrac{2l}{\pi}\right)^2\!\!\left\|\dfrac{\partial f}{\partial x_2}\right\|_{L^2(\Omega)}^2\\
	&+\dfrac{1}{\alpha_3}\left(\dfrac{2l}{\pi}\right)^4\left\|\dfrac{\partial^2 f}{\partial x_1\partial x_2}\right\|_{L^2(\Omega)}^2
	\end{array}
	\end{equation}
	where $\alpha_1$, $\alpha_2$, $\alpha_3$ are positive constants satisfying $$\alpha_1+\alpha_2+\alpha_3=1.$$
	\end{lemma}
\begin{lemma}\cite{Anton1}\label{anton}
Let $\Omega=(0,l)^2 $,  $f\in H^2(\Omega)$  with $f(0,0)=0$,  $\eta>0$. Then 
\begin{equation}
\begin{array}{ll}
\eta \|f\|^2&\le \beta_1\left(\dfrac{2l}{\pi}\right)^2\left\|\dfrac{\partial f}{\partial x_1}\right\|^2+\beta_2\left(\dfrac{2l}{\pi}\right)^2\left\|\dfrac{\partial f}{\partial x_2}\right\|^2
+\beta_3\left(\dfrac{2l}{\pi}\right)^4\left\|\dfrac{\partial^2 f}{\partial x_1 \partial x_2}\right\|^2
\end{array}
\end{equation}
for any $\beta_1$, $\beta_2$, $\beta_3$ satisfying 
\begin{equation}\label{mingbai}
{\rm diag}\{\beta_1,\beta_2,\beta_3\}\geq \eta \left[\begin{array}{ccc}1&1&1\\1&1&1\\1&1&1\end{array}\right].
\end{equation}
	\end{lemma}
	The following version of 2D Sobolev inequality will be useful:	
\begin{lemma}\label{kka}
 Let $\Omega=(0,1)^2$ and 
  $w=w(x_1,x_2)\in H^2(\Omega)\cap H_0^1(\Omega)$, where $(x_1,x_2)\in \Omega$. Then 
\begin{equation}\label{eqiang}
\begin{array}{ll}
\|w\|_{C^0(\bar\Omega)}^2\le 
\dfrac{1}{2}(1+\Gamma)\left[\|w_{x_1}\|^2_{L^2(\Omega)}+\|w_{x_2}\|^2_{L^2(\Omega)}\right]
+\dfrac{1}{\Gamma}\|w_{x_1x_2}\|^2_{L^2(\Omega)},\;\forall \Gamma>0.
\end{array}
\end{equation}
\end{lemma}
\begin{proof} Due to $w\in H_0^1(\Omega)$,
 application of 1D
 Sobolev's inequality to $w$  in $x_2$  yields
 \begin{equation}\label{new2}
 \begin{array}{ll}
 \max\limits_{x_2\in [0,1]}w^2(x_1,x_2)
 \le \disp  \int_0^1\!w_{x_2}^2(x_1,\xi_2) d\xi_2.
\end{array}                                                               
\end{equation}
Further application of 
Lemma 4.1 of \cite{kang} to  $w_{x_2}$ in $x_1$ leads to
  \begin{equation}\label{new4}
 \begin{array}{ll}
 \max\limits_{x_1\in [0,1]}w_{x_2}^2(x_1,\xi_2)\\
 \le \disp  (1+\Gamma)\int_0^{1}\!w_{x_2}^2(\xi_1,\xi_2)d\xi_1 + \dfrac{1}{\Gamma}\int_0^1\!w_{x_1x_2}^2(\xi_1,\xi_2) d\xi_1,\;\forall \Gamma>0.
\end{array}                                                               
\end{equation}
Substitution of \dref{new4} into \dref{new2} yields
\begin{equation}\label{new3}
\begin{array}{ll}
\|w\|_{C^0(\bar\Omega)}^2=\max\limits_{(x_1,x_2)\in \bar\Omega}w^2(x_1,x_2)
=\max\limits_{x_1\in  [0,1]}\left[\max\limits_{x_2\in  [0,1]}w^2(x_1,x_2)\right]\\
\le \max\limits_{x_1\in  [0,1]} \left[ \disp\int_0^1\!w_{x_2}^2(x_1,\xi_2) d\xi_2\right]
\le  \disp\int_0^1\max\limits_{x_1\in  [0,1]}\!w_{x_2}^2(x_1,\xi_2) d\xi_2\\
\le  (1+\Gamma)\disp\int_0^{1}\int_0^{1}\!w_{x_2}^2(\xi_1,\xi_2)d\xi_1d\xi_2+\dfrac{1}{\Gamma}\int_0^1\int_0^1w_{x_1x_2}^2(\xi_1,\xi_2) d\xi_1d\xi_2.
\end{array}
\end{equation}
Following the same procedure, we can obtain 
\begin{equation}\label{new1}
\begin{array}{ll}
\|w\|_{C^0(\bar\Omega)}^2\le  (1+\Gamma)\disp\int_0^{1}\int_0^{1}w_{x_1}^2(\xi_1,\xi_2)d\xi_1d\xi_2+\dfrac{1}{\Gamma}\int_0^1\int_0^1w_{x_1x_2}^2(\xi_1,\xi_2) d\xi_1d\xi_2.
\end{array}
\end{equation}
From \dref{new3} and \dref{new1} it  follows that \dref{eqiang} holds.
\end{proof}

\begin{remark}
  In Lemma \ref{kka}, we give a new 2D Sobolev inequality with constants depending on a free parameter $\Gamma>0$. Lemma \ref{kka} is very useful and  plays an important role in the stability analysis leading to LMIs (with $\Gamma$ as a decision parameter) that guarantee stability and give a bound on the domain of attraction. \end{remark}
	\begin{lemma}(Halanay's Inequality \cite{halanay})
	Let $V:[-h,\infty)\to [0,\infty)$ be an absolutely continuous function. If there exist $0<\delta_1<2\delta$  such that for all $t\geq 0$ the following inequality holds
	\begin{equation}
	\begin{array}{ll}
	\dot V(t)+2\delta V(t)-\delta_1\sup\limits _{-h\le \theta \le 0}V(t+\theta)\le 0,
	\end{array}
	\end{equation}
	then we have
	\begin{equation}
	V(t)\le e^{-2\sigma t} \sup_{-h\le \theta \le 0} V(\theta),\; t\geq 0,
	\end{equation}
	where   
 $\sigma$ is a unique solution of  
 \begin{equation}\label{hdou}
 \sigma=\delta-\dfrac{\delta_1}{2} e^{2\sigma h}.
 \end{equation}
\end{lemma}

\section{Global stabilization: continuous static output-feedback}
In this section, 
we will  establish the well-posedness and stability analysis for the system \dref{a}  under the continuous-time averaged measurements 
\begin{equation}\label{ww}
	\begin{array}{ll}
	y_{j}(t)=\dfrac{\int_{\Omega_{j}}z(x,t)dx}{|\Omega_{j}|},\;
	\;j=1,\dots, N
	\end{array}
	\end{equation}	
	or point measurements 
	\begin{equation}\label{taylor}
	\begin{array}{ll}
	y_{j}(t)=z(\bar x_j,t),\;
	\;j=1,\dots, N
	\end{array}
	\end{equation}
via a controller 
\begin{equation}\label{yy}
	U_{j}(t)=-\mu y_{j}(t),\;\;j=1,\dots, N.
	\end{equation}
The closed-loop system can be represented in the following form:
\begin{equation}\label{b}
	\left\{\begin{array}{ll}
	z_t+zz_{x_1}+ (1-\kappa) z_{x_1x_1}\!-\!\kappa z_{x_2x_2}+ \Delta^2 z\\=-\mu\sum\limits_{j=1}^N \chi_{j}(x)[z-f_{j}], \hspace{0.2cm}\;(x,t)\in \Omega\times [0,\infty),\\
	z|_{\partial \Omega}=0,\;\dfrac{\partial z}{\partial n}|_{\partial \Omega}=0,\\
	z(x,0)=z_0(x),
	\end{array}\right.
	\end{equation}
	where for \dref{ww}
	\begin{equation}\label{dili}
	f_{j}(x,t)=z(x,t)-\dfrac{\int_{\Omega_{j}}z(\zeta,t)d\zeta}{|\Omega_{j}|},
	\end{equation}
	for \dref{taylor}
		\begin{equation}\label{reba}
	f_{j}(x,t)=z(x,t)-z(\bar x_j,t).
	\end{equation}
	Now we establish the well-posedness of the system \dref{b} subject to \dref{dili} or \dref{reba}.
	Define the spatial differential operator $A : D(A) \subset L^2(\Omega)\to L^2(\Omega)$ as follows:
	\begin{equation}
	\left\{\begin{array}{ll}
	Af=-\Delta^2 f, \;\forall f\in D(A),\\
	D(A)=\{f\in H^4(\Omega): f|_{\partial \Omega}=0,\;\dfrac{\partial f}{\partial n}|_{\partial \Omega}=0\}.
	\end{array}\right.
	\end{equation}
	Note that   $A^*=A$ and ${\rm Re}\langle Af, f\rangle \le 0,\; \forall f\in D(A).$
	Thus, the operator $A$ is  self-adjoint and dissipative.  Moreover,  the
inverse $A^{-1}$  is bounded, and hence $0\in \rho(A)$. By the Lumer-Phillips theorem \cite{Pazy},  $A$ generates a $C_0$-semigroup. Since  the resolvent of $A$ is compact on $L^2(\Omega)$, the spectrum of $A$ consists of isolated eigenvalues only, and  a sequence of corresponding eigenfunctions of $A$ forms an orthonormal basis of $L^2(\Omega)$.   Let $\{\lambda_n\}$ be the eigenvalues of $A$ and let $\{\phi_n\}$ be the corresponding eigenfunctions, i.e. $A\phi_n=\lambda_n \phi_n$. Since $A$ is negative, all the eigenvalues are located on the negative real axis,  i.e. $\lambda_n<0$. For any $x_0\in L^2(\Omega)$, it can be presented in the following form:
$x_0=\sum\limits_{n=1}^{\infty} a_n\phi_n$ with  $\sum\limits_{n=1}^\infty |a_n|^2=\|x_0\|_{ L^2(\Omega)}^2$. 
Therefore, for any $\tau\in \mathbb R$, 
 \begin{equation*}
 \begin{array}{ll}
\|(i\tau I-A)^{-1}x_0\|_{ L^2(\Omega)}^2=\sum\limits_{n=1}^\infty  \dfrac{|a_n|^2}{|i\tau-\lambda_n|^2}=\sum\limits_{n=1}^\infty  \dfrac{|a_n|^2}{\tau^2+\lambda_n^2}\\
\le\sum\limits_{n=1}^\infty  \dfrac{|a_n|^2}{\tau^2}=\dfrac{1}{\tau^2}\|x_0\|_{ L^2(\Omega)}^2,
\end{array}
\end{equation*} 
which implies
\begin{equation*}
\|(i\tau I-A)^{-1}\|\le \dfrac{1}{|\tau|},\;\forall \tau \in \mathbb R.
\end{equation*}
Hence, \begin{equation*}
\overline{\lim_{|\tau|\to \infty}}\|\tau(i\tau I-A)^{-1}\|<\infty.
\end{equation*}
 Then from Theorem 1.3.3 of \cite{zyl}, it follows that $A$ generates an analytic semigroup. Since $-A$ is positive, $(-A)^{\frac{1}{2}}$ is also positive and
	$$D((-A)^{\frac{1}{2}})=\{f\in H^2(\Omega): f|_{\partial \Omega}=0,\;\dfrac{\partial f}{\partial n}|_{\partial \Omega}=0\}.$$ 
	The norm of $D((-A)^{\frac{1}{2}})$ is given by 
	$$\|f\|^2_{D((-A)^{\frac{1}{2}})}=\int_\Omega |\Delta f|^2dx.$$
	Throughout the paper, we assume that  $z_0\in D((-A)^{\frac{1}{2}})$. We can rewrite the system \dref{b} subject to \dref{dili} or \dref{reba} as the evolution equation:
	\begin{equation}\label{meme}
	\left\{\begin{array}{ll}
	\dfrac{d}{dt}z(\cdot, t)=Az(\cdot,t)+F(z(\cdot,t)),\\
	z(\cdot,0)=z_0(\cdot)
	\end{array}\right.
	\end{equation}
	subject to 
	\begin{equation*}
	\begin{array}{ll}
	F(z(\cdot,t))&=-z(x,t)z_{x_1}(x,t)-(1-\kappa)z_{x_1x_1}(x,t)\\&+\kappa z_{x_2x_2}(x,t)
	-\mu\sum\limits_{j=1}^N \chi_{j}(x)[z(x,t)-f_{j}(x,t)].
	\end{array}
	\end{equation*}
	Note that the nonlinear term $F$ is locally Lipschitz continuous, that is, there exists a positive constant $l(C)$ such that the following inequality holds:
	$$\|F(z_1)-F(z_2)\|_{L^2(\Omega)}\le l(C)\|z_1-z_2\|_{D((-A)^{\frac{1}{2}})}$$
	for any $z_1$,  $z_2\in D((-A)^{\frac{1}{2}})$ with $\|z_1\|_{D((-A)^{\frac{1}{2}})}\le C$,  $\|z_2\|_{D((-A)^{\frac{1}{2}})}\le C$.
Here we prove the nonlinear term $F$ is locally Lipschitz continuous  for the case of point measurements. From the expression of $F(z(\cdot,t))$,  using Minkowskii's inequality  we have
$$\begin{array}{ll}\|F(z_1)-F(z_2)\|_{L^2(\Omega)}\le \|z_1z_{1x_1}-z_2z_{2x_1}\|_{L^2(\Omega)}
+|1-\kappa|\cdot\|z_{1x_1x_1}-z_{2x_1x_1}\|_{L^2(\Omega)}\\
\hspace{1.2cm}+|\kappa|\cdot\|z_{1x_2x_2}-z_{2x_2x_2}\|_{L^2(\Omega)}+\mu \|z_1-z_2\|_{L^2(\Omega)}\\
\hspace{1.2cm}+\mu \sum \limits_{j=1}^N\|\int_{\bar x_j}^x z_{1\xi}(\xi,t)-z_{2\xi}(\xi,t)d\xi\|_{L^2(\Omega_j)}.
\end{array}\eqno(*)$$
Note that $D((-A)^{\frac{1}{2}})=\{f\in H^2(\Omega): f|_{\partial \Omega}=0,\;\dfrac{\partial f}{\partial n}|_{\partial \Omega}=0\}$. Since $z_1$, $z_2 \in D((-A)^{\frac{1}{2}})$ with $\|z_1\|_{D((-A)^{\frac{1}{2}})}\le C$,  $\|z_2\|_{D((-A)^{\frac{1}{2}})}\le C$, we obtain that 
$z_1-z_2 \in D((-A)^{\frac{1}{2}})$, and there exist some positive constants $M_1$ and $M_2$ such that 
$$
\left\{\begin{array}{ll}\|z_1\|_{L^2(\Omega)}\le M_1\left[\|z_{1x_1}\|_{L^2(\Omega)}+\|z_{1x_2}\|_{L^2(\Omega)}\right], \\
\|z_{1x_1}\|_{L^2(\Omega)}+\|z_{1x_2}\|_{L^2(\Omega)}\le M_2\|z_1\|_{D((-A)^{\frac{1}{2}})},\\
|z_2\|_{L^2(\Omega)}\le M_1\left[\|z_{2x_1}\|_{L^2(\Omega)}+\|z_{2x_2}\|_{L^2(\Omega)}\right], \\
\|z_{2x_1}\|_{L^2(\Omega)}+\|z_{2x_2}\|_{L^2(\Omega)}\le M_2\|z_2\|_{D((-A)^{\frac{1}{2}})},\\
\|z_1\!-\!z_2\|_{L^2(\Omega)}\!\!\le \! M_1\!\left[\|z_{1x_1}\!-\!z_{2x_1}\|_{L^2(\Omega)}\!+\!\|z_{1x_2}\!-\!z_{2x_2}\|_{L^2(\Omega)}\right], \\
\|z_{1x_1}\!\!-\!\!z_{2x_1}\!\|_{L^2(\Omega)}\!+\!\|z_{1x_2}\!\!-\!\!z_{2x_2}\|_{L^2(\Omega)}\!\!\le \!M_2\|z_1\!-\!z_2\|_{D((-A)^{\frac{1}{2}})}.
\end{array}\right.
$$
Hence, the following holds:
\begin{equation*}
\begin{array}{ll}
\|z_1z_{1x_1}-z_2z_{2x_1}\|_{L^2(\Omega)}
=\|z_1z_{1x_1}-z_2z_{1x_1}+z_2z_{1x_1}-z_2z_{2x_1}\|_{L^2(\Omega)}\\
\le \|z_1-z_2\|_{L^2(\Omega)}\|z_{1x_1}\|_{L^2(\Omega)}\!+\!\|z_2\|_{L^2(\Omega)}\|z_{1x_1}-z_{2x_1}\|_{L^2(\Omega)}\\
\le M_1M_2\|z_1-z_2\|_{D((-A)^{\frac{1}{2}}} M_2\|z_1\|_{D((-A)^{\frac{1}{2}}} \\+M_1M_2\|z_2\|_{D((-A)^{\frac{1}{2}}}M_2\|z_1-z_2\|_{D((-A)^{\frac{1}{2}}}\\
\le 2M_1M_2^2C\|z_1-z_2\|_{D((-A)^{\frac{1}{2}}},
\end{array}
\end{equation*}
\begin{equation*}
\begin{array}{ll}
|1-\kappa|\!\cdot\!\|z_{1x_1x_1}-z_{2x_1x_1}\|_{L^2(\Omega)}\!+\!|\kappa|\!\cdot\!\|z_{1x_2x_2}-z_{2x_2x_2}\|_{L^2(\Omega)}\\
\le \max\left\{|1-\kappa|, |\kappa|\right\}\|z_1-z_2\|_{D((-A)^{\frac{1}{2}}},
\end{array}
\end{equation*}
\begin{equation*}
\mu \|z_1-z_2\|_{L^2(\Omega)}\le \mu M_1M_2\|z_1-z_2\|_{D((-A)^{\frac{1}{2}}}.
\end{equation*}
Moreover, Lemma 3 implies that there exist some positive constants $M_3$, $M_4$ and $M_5$ such that
\begin{equation*}
\begin{array}{ll}
\mu \sum \limits_{j=1}^N\|\int_{\bar x_j}^x z_{1\xi}(\xi,t)-z_{2\xi}(\xi,t)d\xi\|_{L^2(\Omega_j)}\\
\le \mu \left[M_3\|z_{1x_1}-z_{2x_1}\|_{L^2(\Omega)}\!+\!M_4\|z_{1x_2}-z_{2x_2}\|_{L^2(\Omega)}\!+\!M_5\|z_{1x_1x_2}-z_{2x_1x_2}\|_{L^2(\Omega)}\right]\\
\le \mu M_3M_2\|z_1-z_2\|_{D((-A)^{\frac{1}{2}}}\!+ \!\mu M_4M_2\|z_1-z_2\|_{D((-A)^{\frac{1}{2}}}\!+ \!\mu M_5 \|z_1-z_2\|_{D((-A)^{\frac{1}{2}}}.
\end{array}
\end{equation*}
Substitution of the above inequalities into the right-hand side  of (*) yields
$$\|F(z_1)-F(z_2)\|_{L^2(\Omega)}\le l(C)\|z_1-z_2\|_{D(-A)^{\frac{1}{2}}},$$
where
$$\begin{array}{ll}l(C)=2M_1M_2^2C+\max\left\{|1-\kappa|, |\kappa|\right\} +\mu (M_1M_2+M_3M_2+M_4M_2+M_5 ).\end{array}$$
	From Theorem 6.3.1 of \cite{Pazy}, it follows that the system \dref{b} subject to \dref{dili} or \dref{reba} has a unique local classical solution $z\in C([0,T), L^2(\Omega))\cap C^1((0,T), L^2(\Omega))$ 
	 for any initial function $z_0\in D((-A)^{\frac{1}{2}})$.
\begin{remark} The above mentioned well-posedness result is dependent on the initial condition $z_0\in D((-A)^{\frac{1}{2}})$. If the initial function $z_0\in L^2(\Omega)$, the solution of the system \dref{b} subject to \dref{dili} or \dref{reba}  may become a mild solution or weak solution.
\end{remark}

	\subsection{Distributed controller under averaged  measurements }
\begin{proposition}
Consider the closed-loop system \dref{b} subject to \dref{dili}.
Given positive scalars $\bar \Delta$, if there exist $\delta>0$,  $\mu>0$  and  
$\lambda_i\geq 0$ $(i=1, 2)$
 such that the following LMI holds:
 \begin{equation}\label{wo}
\begin{array}{ll}
\!\Upsilon& \triangleq \left[ \begin{array}{ccc}-2\mu+2\delta-\lambda_1\dfrac{\pi^2}{2} &\hspace{0.5cm}\Upsilon_{12}&\hspace{0.5cm}\mu \\
\ast &-2&\hspace{0.5cm}0\\
\ast&\ast&\hspace{0.5cm}-\lambda_2\end{array}\right]\le 0
\end{array}
\end{equation}
where  $$\Upsilon_{12}=-\dfrac{\lambda_1}{2}-\lambda_2\dfrac{\bar \Delta^2}{\pi^2}-(1-\kappa),$$
  then the
closed-loop system is globally exponentially stable in the $L^2$-sense:
\begin{equation}\label{kk}
\int_\Omega z^2(x,t)dx\le e^{-2\delta t }\int_\Omega z^2(x,0)dx,\; \forall t\geq 0.
\end{equation}
Furthermore, if the strict LMI \dref{wo}  is feasible for $\delta=0$, then the closed-loop system is exponentially stable with a small enough decay rate.
\end{proposition}	
\begin{proof}
The proof is divided into three parts.\\
Step 1: We have shown that there exists a local clasical solution  to  \dref{b} subject to \dref{dili}, where $T=T(z_0)$. By Theorem 6.23.5 of \cite{kran}, we obtain that the solution exists for any $T>0$ if this solution admits a priori estimate. In Step 3, it will be shown that  the feasibility of LMI \dref{wo} guarantees that the solution of \dref{b} subject to \dref{dili}  admits a priori bound, which can further guarantee the existence of the solution for all $t\geq 0$.

Step 2:  Assume formally that there exists a classical solution of \dref{b} subject to \dref{dili}  for all $t\geq 0$.  We consider 
the following Lyapunov-Krasovskii functional:
\begin{equation}\label{k}
V(t)=\| z(\cdot,t)\|_{L^2(\Omega)}^2.
\end{equation}
Since $z|_{\partial \Omega}=0$ and $\dfrac{\partial z}{\partial n}|_{\partial \Omega}=0$, integration by parts leads to 
\begin{equation}\label{tongtong}
\int_{\Omega} z^2z_{x_1}dx=0,
\end{equation}
\begin{equation}\label{jin}
\begin{array}{ll}
\int_{\Omega} z_{x_1x_2}^2dx=\int_\Omega z_{x_1x_1}z_{x_2x_2}dx.
\end{array}
\end{equation}
Furthermore,  from \dref{jin} it follows that
\begin{equation}\label{jin2}
\begin{array}{ll}
-2 \int_{\Omega} [z_{x_1x_1}^2+z_{x_2x_2}^2+2z_{x_1x_2}^2]dx
=-2 \int_{\Omega} [z_{x_1x_1}+z_{x_2x_2}]^2dx
=-2 \int_{\Omega}|\Delta z|^2dx.
\end{array}
\end{equation}
Differentiating \dref{k} along \dref{b}, integrating by part and using \dref{tongtong}, \dref{jin2} we obtain
\begin{equation}\label{henghengshu}
\begin{array}{ll}
\dot V(t)+2\delta V(t)=2\int_\Omega zz_tdx+2\delta \int_\Omega z^2dx\\

=2(1-\kappa)\int_\Omega z_{x_1}^2dx-2\kappa \int_\Omega z_{x_2}^2dx-2\int_\Omega |\Delta z|^2dx\\
-(2\mu-2\delta) \int_\Omega z^2dx+2\mu \sum\limits_{j=1}^N \int_{\Omega_{j} }zf_{j} dx.
\end{array}
\end{equation}
From Lemma \ref{d}, the Wirtinger's inequality yields 
\begin{equation}\label{kk1}
\begin{array}{ll}
\lambda_1\left [\int_\Omega \nabla z^T\nabla zdx -\frac{\pi^2}{2}\int_\Omega z^2 dx\right]\geq 0,
\end{array}
\end{equation}
where $\lambda_1\geq 0$.\\
For the case of averaged measurements, $f_j$ is given by \dref{dili}. 
Since $\int_{\Omega_{j}}f_{j}(x,t) dx=0$, from Lemma \ref{d}, the Poincar\' e inequality leads to 
\begin{equation*}
\begin{array}{ll}
\int_{ \Omega_{j}}f_{j}^2dx \le \frac{2\bar \Delta^2}{\pi^2}\int_{ \Omega_{j}} \nabla z^T \nabla z dx.
\end{array}
\end{equation*}
Hence,
\begin{equation}\label{mm}
\begin{array}{ll}
\lambda_2 \sum\limits_{j=1}^N\left[\frac{2\bar \Delta^2}{\pi^2}\int_{ \Omega_{j}}\nabla z^T \nabla z dx- \int_{ \Omega_{j}} f_{j}^2dx\right]\geq0,
\end{array}
\end{equation}
where $\lambda_2\geq 0$.\\
Integration by parts yields
\begin{equation}\label{ee}
\begin{array}{ll}
-\int_{\Omega}\nabla z^T\nabla z dx=\int_{\Omega}z\Delta zdx.
\end{array}
\end{equation}
Applying S-procedure \cite{Yakubovich}, we add to $\dot V(t)+2\delta V(t)$ the left-hand side of \dref{kk1}, \dref{mm}  and use \dref{ee}. Then it follows that
\begin{equation*}
\begin{array}{ll}
\dot V(t)+2\delta V(t)
\le \dot V(t)+2\delta V(t)+\lambda_1\left [\int_{\Omega}\nabla z^T\nabla zdx -\frac{\pi^2}{2}\int_\Omega z^2 dx \right]\\
+\lambda_2\left[ \frac{2\bar \Delta^2}{\pi^2}\int_{\Omega}\nabla z^T \nabla z dx- \sum\limits_{j=1}^N\int_{\Omega_j}f_{j}^2dx\right]\\
\le 2(1-\kappa)\int_\Omega z_{x_1}^2dx-2\kappa \int_\Omega z_{x_2}^2dx-2\int_\Omega |\Delta z|^2dx\\
-(2\mu-2\delta+\lambda_1\frac{\pi^2}{2}) \int_\Omega z^2dx+2\mu \sum\limits_{j=1}^N \int_{\Omega_{j} }zf_{j} dx\\
-(\lambda_1+\lambda_2\frac{2\bar \Delta^2}{\pi^2})\int_\Omega z\Delta z dx-\lambda_2\sum\limits_{j=1}^N\int_{\Omega_{j}} f_{j}^2dx.
\end{array}
\end{equation*}
Note that
\begin{equation}\label{rami1}
 \begin{array}{ll}
2(1-\kappa)\int_\Omega z_{x_1}^2dx-2\kappa \int_\Omega z_{x_2}^2dx
= 2(1-\kappa)\int_\Omega \nabla z^T \nabla z dx -2 \int_\Omega z_{x_2}^2dx\\
\le 2(1-\kappa)\int_\Omega \nabla z^T \nabla z dx.
\end{array}
\end{equation}
Then substitution of \dref{ee} into \dref{rami1} yields
\begin{equation}
\begin{array}{ll}
2(1-\kappa)\int_\Omega z_{x_1}^2dx-2\kappa \int_\Omega z_{x_2}^2dx
\le  -2(1-\kappa)\int_\Omega z\Delta z dx.
\end{array}
\end{equation}
Hence,
 \begin{equation*}
\dot V(t)+2\delta V(t)\le \sum_{j=1}^N \int_{\Omega_j}\! \left[ \begin{array}{ccc}z&\Delta z&f_j \end{array}\right] \!\Upsilon \!\left[ \begin{array}{c}z\\ \Delta z\\f_j \end{array}\right]  dx\le 0
\end{equation*}
if $\Upsilon\le 0$ holds.
Therefore, $$V(t) \le e^{-2\delta t} V(0),\; \forall t\geq 0.$$
Note that the feasibility of the strict LMI \dref{wo}  with $\delta = 0$ implies its feasibility with a small enough $\delta_0 > 0$. Therefore, if the strict LMI \dref{wo} holds for $\delta = 0$, then the closed-loop system is exponentially stable with a small decay rate $\delta_0 > 0$.

Step 3: The feasibility of LMI \dref{wo} yields that the solution of \dref{b} subject to \dref{dili}  admits a priori estimate 
$V(t) \le e^{-2\delta t} V(0)$. By Theorem 6.23.5 of \cite{kran}, continuation of this solution under a priori bound to entire interval $[0,\infty)$. 
\end{proof}
\subsection{Distributed controller under  point measurements }
\begin{proposition}
	Consider the closed-loop system \dref{b} subject to \dref{reba}.
	Given positive scalars $\bar \Delta$, if there exist $\delta>0$, $\mu>0$,  $\eta>0$,	$\lambda_1\geq 0$, $\lambda_2\in \mathbb R$ and  
	$\beta_i>0$ $(i=1, 2,3) $
	such that \dref{mingbai} is satisfied and  the following LMIs hold:
	\begin{equation}\label{weiweihe}
	2(1-\kappa)+\beta_1\left(\!\frac{\bar \Delta}{\pi}\right)^2+\lambda_1-\lambda_2\le 0,
	\end{equation}
	\begin{equation}
	-2\kappa+\beta_2\left(\!\frac{\bar \Delta}{\pi}\right)^2+\lambda_1-\lambda_2\le 0,
	\end{equation}
	\begin{equation}\label{woa}
\Lambda\!\!=\!\!\left[\begin{array}{cccc}-2\mu+2\delta-\lambda_1\dfrac{\pi^2}{2}&-\dfrac{\lambda_2}{2}&-\dfrac{\lambda_2}{2}&\mu\\
\ast&-2 &-2+\dfrac{\beta_3}{2}\left(\!\frac{\bar \Delta}{\pi}\right)^4&0\\
\ast&\ast&-2&0\\
\ast&\ast&\ast&-\eta \end{array}\right]\!\!\le \! 0,
	\end{equation} 
	 the closed-loop system is globally exponentially stable satisfying  \dref{kk}.
	Furthermore, if the strict LMI \dref{woa} is feasible for $\delta=0$, then the closed-loop system is exponentially stable with a small enough decay rate.
\end{proposition}	
\begin{proof}
Step 1: We have shown that there exists a local clasical solution  to  \dref{b} subject to \dref{reba}, where $T=T(z_0)$. By Theorem 6.23.5 of \cite{kran}, we obtain that the solution exists for any $T>0$ if this solution admits a priori estimate. In Step 3, it will be shown that  the feasibility of LMIs \dref{weiweihe}-\dref{woa} guarantees that the solution of \dref{b} subject to \dref{reba}  admits a priori bound, which can further guarantee the existence of the solution for all $t\geq 0$.

Step 2:	Assume formally that there exists a classical solution of \dref{b} subject to  \dref{reba} for all $t\geq 0$. Consider $V$ given by \dref{k}.  Differentiating $V$ along \dref{b} and integrating by parts, we have \dref{henghengshu}. \\
 For the case of point measurements, $f_j$ is given by \dref{reba}.
From Lemma \ref{anton}, we have
\begin{equation*}
\begin{array}{ll}
\eta\|f_j\|_{L^2(\Omega_j)}^2\!\!& \le \beta_1\left(\!\frac{\bar \Delta}{\pi}\right)^2\!\!\|z_{ x_1}\|_{L^2(\Omega_j)}^2\! +\!\beta_2\left(\!\frac{\bar \Delta}{\pi}\right)^2\!\|z_{ x_2}\|_{L^2(\Omega_j)}^2
\!+\!\beta_3\left(\frac{\bar \Delta}{\pi}\right)^4\!\|z_{x_1 x_2}\|_{L^2(\Omega_j)}^2
\end{array}
\end{equation*}
for any  scalars $\beta_1$, $\beta_2$, $\beta_3$ such that \dref{mingbai} holds.\\
Hence, 
\begin{equation}\label{yingying}
\begin{array}{ll}
\sum\limits_{j=1}^{N}&\left[\beta_1\left(\!\frac{\bar \Delta}{\pi}\right)^2\!\!\!\|z_{ x_1}\|_{L^2(\Omega_j)}^2\! +\beta_2\left(\!\frac{\bar \Delta}{\pi}\right)^2\!\|z_{ x_2}\|_{L^2(\Omega_j)}^2\right.\\&\left.
+\beta_3\left(\frac{\bar \Delta}{\pi}\right)^4\!\!\|z_{x_1 x_2}\|_{L^2(\Omega_j)}^2-\eta\|f_j\|_{L^2(\Omega_j)}^2\right]\geq 0.
\end{array}
\end{equation}
From   \dref{ee}, for any $\lambda_2\in \mathbb R$ we have 
\begin{equation}
\lambda_2\left[\int_{\Omega}\nabla z^T\nabla z dx+\int_{\Omega}z\Delta zdx\right]=0.
\end{equation}
Similarly, we add to $\dot V(t)+2\delta V(t)$ the left-hand side of \dref{kk1} and \dref{yingying}. Then by taking into account \dref{henghengshu}, we obtain
\begin{equation*}
\begin{array}{ll}
\hspace{0.28cm}\dot V(t)+2\delta V(t)\\
\le \dot V(t)+2\delta V(t)+\lambda_1\left [\int_{\Omega}\nabla z^T\nabla zdx -\frac{\pi^2}{2}\int_\Omega z^2 dx \right]\\
+\lambda_2\left[-\int_{\Omega}\nabla z^T\nabla z dx-\int_{\Omega}z\Delta zdx\right]\\
+\left[\beta_1\left(\frac{\bar \Delta}{\pi}\right)^2\|z_{ x_1}\|_{L^2(\Omega_j)}^2+\beta_2\left(\frac{\bar \Delta}{\pi}\right)^2\|z_{ x_2}\|_{L^2(\Omega_j)}^2\right.\\\left.
+\beta_3\left(\frac{\bar \Delta}{\pi}\right)^4\|z_{x_1 x_2}\|_{L^2(\Omega_j)}^2-\eta\|f_j\|_{L^2(\Omega_j)}^2\right]\\
\le \left[2(1-\kappa)+\beta_1\left(\frac{\bar \Delta}{\pi}\right)^2+\lambda_1-\lambda_2\right]\int_\Omega z_{x_1}^2dx\\
+\left[-2\kappa+\beta_2\left(\frac{\bar \Delta}{\pi}\right)^2+\lambda_1-\lambda_2\right]  \int_\Omega z_{x_2}^2dx\\
+\beta_3\left(\!\frac{\bar \Delta}{\pi}\right)^4\int_\Omega z_{x_1x_2}^2dx-2\int_\Omega |\Delta z|^2dx\\
-(2\mu-2\delta+\lambda_1\frac{\pi^2}{2}) \int_\Omega z^2dx+2\mu \sum\limits_{j=1}^N \int_{\Omega_{j} }zf_{j} dx\\
-\lambda_2\int_{\Omega}z\Delta zdx-\eta\sum\limits_{j=1}^N\int_{\Omega_{j}} f_{j}^2dx.
\end{array}
\end{equation*}
By using \dref{jin} and \dref{jin2}, the following inequality holds for all $t\geq 0$
\begin{equation}
\dot V(t)+2\delta V(t) \le \sum\limits_{j=1}^N\int_{\Omega_{j}}\psi^T(x,t)\Lambda \psi(x,t) dx
\end{equation}
where 
$$\psi(x,t)={\rm col}\{z, z_{x_1x_1},z_{x_2x_2},f_j \}.$$
Therefore, the LMIs   \dref{weiweihe}-\dref{woa} yield \dref{kk}.

Step 3: The feasibility of LMIs \dref{weiweihe}-\dref{woa} yields that the solution of \dref{b} subject to \dref{reba}  admits a priori estimate 
$V(t) \le e^{-2\delta t} V(0)$. By Theorem 6.23.5 of \cite{kran}, continuation of this solution under a priori bound to entire interval $[0,\infty)$. 
\end{proof}

\section{Sampled-data regional stabilization}
\subsection{Sampled-data control under averaged measurements}
For $j= 1, \cdots , N$; $k = 0, 1,\cdots$ we consider the quantities
\begin{equation}\label{gaozhong}
f_{j}(x,t)=z(x,t)-\dfrac{\int_{\Omega_{j}}z(\zeta,t)d\zeta}{|\Omega_{j}|},
\end{equation}
\begin{equation}\label{chuzhong}
g_j(t)=\dfrac{1}{t-t_k}\dfrac{\int_{\Omega_{j}}\int_{t_k}^tz_s(\zeta,s)ds d\zeta}{|\Omega_{j}|}.
\end{equation}
Then  the controller \dref{y} subject to \dref{w} leads to the  closed-loop
system
\begin{equation}\label{niu}
	\left\{\begin{array}{ll}
	z_t\!+zz_{x_1}\!+ (1-\kappa) z_{x_1x_1}\!-\!\kappa z_{x_2x_2}+ \Delta^2 z
	=-\mu\sum\limits_{j=1}^N \chi_{j}(x)[z-f_{j}-(t-t_k)g_j], \\
	\hspace{8cm}(x,t)\in \Omega\times [t_k,t_{k+1}),\\
	z|_{\partial \Omega}=0,\;\dfrac{\partial z}{\partial n}|_{\partial \Omega}=0.
	\end{array}\right.
	\end{equation}

	
Now we  use the step method (see e.g. \cite{Bellman,Emilia4}) to establish the proof of the well-posedness for system \dref{niu}.	
For $t\in [t_0,t_1]$, we consider the following equation:
\begin{equation}\label{t0}
	\left\{\begin{array}{ll}
	z_t+zz_{x_1}+ (1-\kappa) z_{x_1x_1}\!-\!\kappa z_{x_2x_2}+ \Delta^2 z
	=-\mu\sum\limits_{j=1}^N \chi_{j}(x)\dfrac{\int_{\Omega_j}z_0(x)dx}{|\Omega_j|}, \\
	z|_{\partial \Omega}=0,\;\dfrac{\partial z}{\partial n}|_{\partial \Omega}=0.
	\end{array}\right.
	\end{equation}
	Then system \dref{t0} can be represented as an evolution equation \dref{meme} subject to 
	\begin{equation*}
	\begin{array}{ll}
	\!F(z(\cdot,t))\!&\!=-z(x,t)z_{x_1}(x,t)\!-(1-\kappa)z_{x_1x_1}(x,t)\!+\kappa z_{x_2x_2}(x,t)\\&-\mu\sum\limits_{j=1}^N \chi_{j}(x)\dfrac{\int_{\Omega_j}z_0(x)dx}{|\Omega_j|}.
	\end{array}
	\end{equation*}
	Note that the nonlinearity $F(z(\cdot,t))$ is locally Lipschitz continuous. From Theorem 3.3.3 of \cite{Henry}, it follows that 
	there exists a unique local strong solution $z(\cdot,t)\in C([0,T];D((-A)^{\frac{1}{2}}))\cap C^1((0,T];D(A))$ of \dref{t0} initialized with $z_0\in D((-A)^{\frac{1}{2}})$ 
        on some interval $[0,T]\subset [0,t_1]$, where $T=T(z_0)>0$. By Theorem 6.23.5 of \cite{kran}, we obtain that if this solution admits a priori estimate, then the solution exists on the entire $[0,t_1]$. The priori estimate on the solutions starting from the domain of attraction 
        will be guaranteed by the stability conditions that we will provide (see Theorem \ref{beibei}).
        Then we apply the same line of reasoning step-by-step to the time segments $[t_1,t_2]$, $[t_2,t_3]$, $\cdots$. Following this procedure, we find that the strong solution exists for all $t\geq 0$. 

In order to derive the stability conditions for \dref{niu} we employ
the following Lyapunov-Krasovskii functional
\begin{equation}\label{jingjing}
\begin{array}{ll}
V_1(t)&=p_1\|z\|^2_{L^2(\Omega)} +p_2\|\Delta z\|^2_{L^2(\Omega)}\\
&+r(t_{k+1}-t) \int_{\Omega} \int_{t_k}^t e^{2\delta (s-t)}z_s^2(x,s)dsdx,\\
&t\in [t_k,t_{k+1}),\; p_1>0,\; p_2>0,\; r>0.
\end{array}
\end{equation}
\begin{remark}
Note that without delay/sampling behavior, the energy norm is usually used. In the present work, due to the sampling terms, we need to use Lyapunov-Krasovskii functionals (see e.g. \cite{Emilia2,emilia5}). Therefore, additionally to  the energy norm $p_1\|z\|^2_{L^2(\Omega)} +p_2\|\Delta z\|^2_{L^2(\Omega)}$, we employ the term $r(t_{k+1}-t) \int_{\Omega} \int_{t_k}^t e^{2\delta (s-t)}z_s^2(x,s)dsdx$  to deal with the sampling. 
\end{remark}
For convenience we define $$V=D((-A)^{\frac{1}{2}})$$ with the norm
\begin{equation*}
\|z\|_V^2=p_1\|z\|^2_{L^2(\Omega)} +p_2\|\Delta z\|^2_{L^2(\Omega)}.
\end{equation*}
Here $p_1$ and $p_2$ are positive constants that are related to the Lyapunov-Krasovskii
functional \dref{jingjing}. By using Lyapunov-Krasovskii functional \dref{jingjing}, in
Theorem \ref{beibei} we provide LMI conditions for regional exponential
stability of \dref{niu} and for a bound on the domain of attraction.
\begin{theorem}\label{beibei}
Consider the closed-loop system \dref{niu}.  Given positive scalars $C$, $h$, $\mu$, $\bar \Delta$ and $\delta$, let there exist
scalars $r>0$, $\Gamma>0$, $p_1>0$, $p_2>0$, $\lambda_i\geq 0$ $(i=1,2)$ and $\lambda_3\in \mathbb R$ satisfy the linear matrix inequalities:
\begin{equation}\label{yangyang}
\Xi_i|_{z=C}<0,\; \Xi_i|_{z=-C}<0,\; i=1,2
\end{equation}
\begin{equation}\label{ninian}
 \left[\begin{array}{cc} p_2-(1+\Gamma)\dfrac{1}{\pi^2}&\hspace{0.3cm} \sqrt{\dfrac{1}{2}} \\ \ast&\hspace{0.3cm}  \Gamma \end{array}\right]> 0,
\end{equation}
where  
\begin{equation}\label{uu}
\Xi_1=\begin{pmat}[{......|}]
 &&&&&&& -r h z\cr
 &&&&&&&0\cr
 &&&&{ {\Phi_1}}&&&-(1-\kappa)rh\cr
 &&&&&&&\kappa rh\cr
 &&&&&&& -rh\cr
  &&&&&&& -\mu rh\cr
 &&&&&&& \mu rh\cr\-
   &&&&{{\ast} } & && -rh\cr   
  \end{pmat},
\end{equation}
\begin{equation}\label{qq}
\Xi_2=\begin{pmat}[{......|}]
 &&&&&&& -r h z\cr
 &&&&&&&0\cr
 &&&&{ {\Phi_2}}&&&-(1-\kappa)rh\cr
 &&&&&&&\kappa rh\cr
 &&&&&&& -rh\cr
  &&&&&&& -\mu rh\cr
 &&&&&&& \mu rh\cr
 &&&&&&& \mu rh^2 \cr\-
   &&&&{{\ast} } & && -rh\cr   
  \end{pmat},
\end{equation}

$\Phi_1=\{\phi_{ij}\}$ is a 
symmetric matrix composed from
$$
\begin{array}{ll}
\phi_{11}=2p_1(1-\kappa)+\lambda_1+\dfrac{2\bar \Delta^2}{\pi^2}\lambda_2-\lambda_3,\;
\phi_{15}=-p_2z,\\
\phi_{22}=-2p_1\kappa+\lambda_1+\dfrac{2\bar \Delta^2}{\pi^2}\lambda_2-\lambda_3,\\
\phi_{33}=-2p_1+2\delta p_2,\; \phi_{34}=2\delta p_2,\\ \phi_{35}=-p_2(1-\kappa),\phi_{36}=-\dfrac{\lambda_3}{2},\\
\phi_{44}=-2p_1+2\delta p_2,\; \phi_{45}=p_2\kappa, \; \phi_{46}=-\dfrac{\lambda_3}{2},\\
\phi_{55}=-2p_2,\;\phi_{56}=-p_2\mu,\; \phi_{57}=p_2\mu,\\
\phi_{66}=-2p_1\mu+2\delta p_1-\dfrac{\pi^2}{2}\lambda_1, \;
\phi_{67}=p_1\mu,\\
\phi_{77}=-\lambda_2,
\end{array}
$$
\begin{equation}
\Phi_2=\begin{pmat}[{......|}]
    &&&&{ {\Phi_1}} & &&\ast \cr\-
    0&0&0&0&p_2\mu h&p_1\mu h&0&-rhe^{-2\delta h}\cr
  \end{pmat}.
\end{equation}
Then for any initial state $z_0\in V$ satisfying $\|z_0\|_V^2<C^2$,  a unique solution of  \dref{niu} exists and  satisfies
$$\begin{array}{ll} p_1\|z\|^2_{L^2(\Omega)} +p_2\|\Delta z\|^2_{L^2(\Omega)}\\\le e^{-2\delta t} [p_1\|z_0\|^2_{L^2(\Omega)} +p_2\|\Delta z_0\|^2_{L^2(\Omega)}],\; , t\ge 0.
\end{array}$$
Furthermore, if the strict LMI \dref{yangyang} is feasible for $\delta= 0$, then the closed-loop system is exponentially stable with a small enough decay rate.
\end{theorem}
\begin{proof}
The proof is divided into three parts.\\
Step 1: We have shown that there exists a unique local strong solution  to \dref{niu}  on some interval 
$[0,T]\subset [0,t_1]$. By Theorem 6.23.5 of \cite{kran}, we obtain that the solution exists on the entire interval $[0,t_1]$ if this solution admits a priori estimate. In Step 3, it will be shown that the feasibility of LMIs \dref{yangyang}, \dref{ninian} guarantees that the solution  of \dref{niu} admits  a priori bound, which can further guarantee the existence of the solution for all $t\geq 0$.

Step 2: Assume formally that there exists a solution of \dref{niu} for all $t\geq 0$. Differentiating $V_1$ along \dref{niu}, we have
\begin{equation}\label{yanzi}
\begin{array}{ll}
\dot V_1(t)+2\delta V_1(t)
&=2p_1\int_\Omega zz_tdx+2p_2\int_\Omega \Delta z \Delta z_tdx\\
&-r\int_{\Omega} \int_{t_k}^t e^{2\delta (s-t)}z_s^2(x,s)dsdx \\
&+r(t_{k+1}-t)\int_{\Omega}z_t^2(x,t)dx\\
&+2\delta p_1 \int_\Omega z^2dx+2\delta p_2 \int_\Omega |\Delta z|^2dx.
\end{array}
\end{equation}
Substitution of $z_t$
from \dref{niu}  leads to
\begin{equation}\label{kkw1}
\begin{array}{ll}
r(t_{k+1}-t)\int_{\Omega}z_t^2(x,t)dx\\
=r(t_{k+1}\!-t)\sum\limits_{j=1}^N\int_{\Omega_j}[-zz_{x_1}-(1-\kappa) z_{x_1x_1}+\kappa z_{x_2x_2}\\
- \Delta^2 z-\mu z+\mu f_j+ \mu (t-t_k)g_j]^2dx\\
\le rh\sum\limits_{j=1}^N\int_{\Omega_j} [-zz_{x_1}-\!(1-\kappa) z_{x_1x_1}\!+\kappa z_{x_2x_2}- \Delta^2 z\\
-\mu z+\mu f_j+ \mu (t-t_k)g_j]^2dx.
\end{array}
\end{equation}	
Integration by parts yields
\begin{equation}\label{erb}
\begin{array}{ll}
2p_1\int_\Omega zz_tdx
=2p_1(1-\kappa)\int_\Omega z_{x_1}^2dx-2p_1\kappa \int_\Omega z_{x_2}^2dx-2p_1\int_\Omega |\Delta z|^2dx\\
-2p_1\mu \int_\Omega z^2dx+2p_1\mu \sum\limits_{j=1}^N \int_{\Omega_{j} }z[f_{j}+ (t-t_k)g_j]dx,
\end{array}
\end{equation}
and 
\begin{equation}
\begin{array}{ll}
2p_2\int_\Omega \Delta z \Delta z_tdx
=2p_2\int_\Omega \Delta^2 z \cdot  z_tdx\\
=2p_2\int_\Omega \Delta ^2 z [-zz_{x_1} -(1-\kappa)z_{x_1x_1}+\kappa z_{x_2x_2}-\Delta^2 z]dx\\
-2p_2\mu \int_\Omega\Delta ^2 z [z-f_j-(t-t_k)g_j]dx.
\end{array}
\end{equation}
The Jensen inequality leads to 
\begin{equation}
\begin{array}{ll}
-r\int_{\Omega} \int_{t_k}^t e^{2\delta (s-t)}z_s^2(x,s)dsdx
\le -r e^{-2\delta h}\int_{\Omega}\frac{1}{t-t_k}\left[\int_{t_k}^{t} z_s(x,s)ds\right]^2dx\\
\le -r (t-t_k)e^{-2\delta h} \sum\limits_{j=1}^N \int_{\Omega_j} g_j^2dx.
\end{array}
\end{equation}		
From \dref{ee}, we obtain
\begin{equation}\label{een}
\begin{array}{ll}
\lambda_3\left[-\int_{\Omega}\nabla z^T \nabla z dx-\int_{\Omega}z\Delta zdx\right]=0,
\end{array}
\end{equation}
where $\lambda_3\in \mathbb R$.\\
Set $$\begin{array}{ll}
 \eta_1 ={\rm col}\{z_{x_1}, z_{x_2}, z_{x_1x_1}, z_{x_2x_2}, \Delta^2 z, z, f_j\},\\
\eta_2 ={\rm col}\{  \eta_1,g_j\},\\
\beta \triangleq\left[\begin{array}{ccccccccc} -z&0& -(1-\kappa) & \kappa &-1&-\mu&\mu& \mu(t-t_k) \end{array}\right].
\end{array}
$$
	Then
\begin{equation}\label{kkw3}
\begin{array}{ll}
[-zz_{x_1}-\!(1-\kappa) z_{x_1x_1}\!+\kappa z_{x_2x_2}- \Delta^2 z
-\mu z+\mu f_j+ \mu (t-t_k)g_j]^2
= \eta_2^T\beta^T\beta\eta_2.
\end{array}
\end{equation}
 From \dref{kkw1} and \dref{kkw3} we have
\begin{equation}\label{ppy}
\begin{array}{ll}
r(t_{k+1}-t)\int_{\Omega}z_t^2(x,t)dx\le rh \sum\limits_{j=1}^N \disp\int_{\Omega_j}\eta_2^T\beta^T\beta\eta_2dx.
\end{array}
\end{equation}
Applying S-procedure, we add to $\dot V_1(t)+2\delta V_1(t)$ the left-hand side of \dref{kk1}, \dref{mm}, \dref{een}. Then,
\begin{equation}\label{sihuanghu}
\begin{array}{ll}
\dot V_1(t)+2\delta V_1(t)\\
\le \dot V_1(t)+2\delta V_1(t)+\lambda_1\left [\int_{\Omega} \nabla z^T \nabla z dx -\frac{\pi^2}{2}\int_\Omega z^2 dx\right]\\
+\lambda_2\left[ \frac{2\bar \Delta^2}{\pi^2}\int_{\Omega}\nabla z^T\nabla z dx- \sum\limits_{j=1}^N\int_{\Omega_{j}}f_{j}^2dx\right]\\
+\lambda_3\left[-\int_{\Omega}\nabla z^T \nabla z dx-\int_{\Omega}z\Delta zdx\right]\\
\le \sum\limits_{j=1}^N \disp\int_{\Omega_j} \dfrac{h-t+t_k}{h}\eta_1^T\Phi_1  \eta_1+\dfrac{t-t_k}{h}\eta_2^T\Phi_2 \eta_2 dx\\
+rh\sum\limits_{j=1}^N \disp\int_{\Omega_j}\eta_2^T\beta^T\beta\eta_2dx
-(4p_1-4\delta p_2)\int_{\Omega} z_{x_1x_2}^2dx.
\end{array}
\end{equation}
Note that LMIs \dref{yangyang} imply that $\phi_{33}<0$, i.e. $p_1> \delta p_2$.
As in \cite{Anton},  first  we assume that 
\begin{equation}\label{lele}
\|z(\cdot,t)\|_{C^0(\bar\Omega)}<C,\; \forall t\geq 0.
\end{equation}
Note that $\dfrac{h-t+t_k}{h}+\dfrac{t-t_k}{h} =1$ and $t-t_k\le h$.
Under the assumption \dref{lele},  
applying Schur complement  to \dref{ppy},  from \dref{kkw3}-\dref{sihuanghu} we obtain 
\begin{equation}\label{rami}
\begin{array}{ll}
\dot V_1(t)+2\delta V_1(t)\\
\le\sum\limits_{j=1}^N \disp\int_{\Omega_j} \dfrac{h-t+t_k}{ h} \left[\begin{array}{cc}\eta_1^T& 1\end{array}\right]\Xi_1\left[\begin{array}{c}\eta_1
\\1\end{array}\right]dx\\+\sum\limits_{j=1}^N \disp\int_{\Omega_j}\dfrac{t-t_k}{ h}
[\begin{array}{cc}\eta_2^T& 1\end{array}]\Xi_2\left[\begin{array}{c}\eta_2\\1\end{array}\right]dx
\le 0
\end{array}
\end{equation}
if $\Xi_1<0$, $\Xi_2<0$ for all  $z\in (-C,C)$.\\
Matrices  $\Xi_1$ and $\Xi_2$ given by \dref{uu}, \dref{qq} are affine in $z$.
Hence,  $ \Xi_1< 0$ and $ \Xi_2< 0$ for all $z \in  (-C, C)$ if these inequalities
hold in the vertices $z=\pm C$ hold, i.e. if LMIs \dref{yangyang} are feasible.

We prove next that \dref{lele} holds. Lemma \ref{kka} and   Schur complement theorem lead to 
\begin{equation}\label{tingtingmeng}
\begin{array}{ll}
\!\!\|z\|_{C^0(\bar\Omega)}^2\!\!\le \!\dfrac{1}{2}(1+\Gamma)\!\!\left[\|z_{x_1}\|^2_{L^2(\Omega)}\!
+\!\|z_{x_2}\|^2_{L^2(\Omega)}\right]
\!\!+\!\dfrac{1}{\Gamma}\!\|z_{x_1x_2}\|^2_{L^2(\Omega)}\\
\le \left[(1+\Gamma)\dfrac{1}{\pi^2} +\dfrac{1}{2\Gamma}\right]\|\Delta z\|^2_{L^2(\Omega)}
\le V_1(t).
\end{array}
\end{equation}
The last inequality in  \dref{tingtingmeng} follows from \dref{ninian}, and for the second inequality  in  \dref{tingtingmeng} we use 
 the Wirtinger's inequality, \dref{jin} and \dref{jin2}.
Therefore, it is
sufficient to show that
\begin{equation}\label{qqp}
V_1(t)<C^2,\;\forall t\geq 0.
\end{equation}
Indeed, for $t = 0$, the inequality    \dref{qqp} holds. Let \dref{qqp} be false
for some $t_1$. Then $V_1(t_1)\geq C^2>V_1(0)$. Since $V_1$ is continuous in time, there must exist $t^*\in (0,t_1]$ such that
\begin{equation}\label{vv}
V_1(t)<C^2\;\forall t\in [0,t^*)\; and  \; V_1(t^*)=C^2.
\end{equation}
The first relation of \dref{vv}, together with the feasibility of \dref{yangyang}, guarantees that $\dot V_1(t)+2\delta V_1(t)\le 0$
on $[0,t^*)$. Therefore, $V_1(t^*) \le V_1(0) <C^2$. This 
contradicts the second relation of \dref{vv}. Thus, \dref{qqp} and consequently, \dref{rami} is true, which implies provided that $\|z_0\|_V <C$. \\
Note that the feasibility of LMI \dref{yangyang} with $\delta= 0$ implies its feasibility with a small enough $\delta_0> 0$. Therefore, if LMI \dref{yangyang} holds for $\delta = 0$, then the closed-loop system is exponentially stable with a small decay rate.

Step 3: The feasibility of LMIs \dref{yangyang}, \dref{ninian} yields that the solution of \dref{niu}  admits a priori estimate 
$V_1(t) \le e^{-2\delta t} V_1(0)$. By Theorem 6.23.5 of \cite{kran},  this solution (under a priori bound) can be continued to entire interval $[0,\infty)$. 
\end{proof}
\subsection{Sampled-data control under point measurements}
Under the controller \dref{y} subject to \dref{swift}, the closed-loop 
system becomes
\begin{equation}\label{YAZI}
\left\{\begin{array}{ll}
z_t+zz_{x_1}+ (1-\kappa) z_{x_1x_1}\!-\!\kappa z_{x_2x_2}+ \Delta^2 z\\
=-\mu\sum\limits_{j=1}^N \chi_{j}(x)z(\bar x_j,t_k), \;
(x,t)\in \Omega\times [t_k,t_{k+1}),\\
z|_{\partial \Omega}=0,\;\dfrac{\partial z}{\partial n}|_{\partial \Omega}=0.
\end{array}\right.
\end{equation}
\begin{theorem}\label{woy}
Consider the closed-loop system \dref{YAZI}.  Given positive scalars $C$, $h$, $\mu$, $\bar \Delta$ and $\delta_1<2\delta$, let there exist
scalars $r>0$, $\Gamma>0$, $p_1>0$, $p_2>0$, $\eta>0$, $\lambda_i\geq 0$ $(i=1,2)$, $\lambda_3\in \mathbb R$  and  
$\beta_i>0$ $(i=1, 2,3) $
such that \dref{mingbai} is satisfied and  the following LMIs hold:	
\begin{equation}\label{jinhan}
-2\delta_1p_2+\beta_3\left(\frac{\bar\Delta}{\pi}\right)^4\le 0,
\end{equation}
\begin{equation}
\bar \Theta=\left[\begin{array}{ccc} -\delta_1p_2&-\dfrac{\beta_1}{2}\left(\frac{\bar \Delta}{\pi}\right)^2&-\dfrac{\beta_2}{2}\left(\frac{\bar \Delta}{\pi}\right)^2\\
\ast&-\delta_1p_2&0\\
\ast&\ast&-\delta_1p_2\end{array}\right]\le 0,
\end{equation}
\begin{equation}\label{hans}
\left[\begin{array}{cc} p_2-(1+\Gamma)\dfrac{1}{\pi^2}&\hspace{0.3cm} \sqrt{\dfrac{1}{2}} \\ \ast&\hspace{0.3cm}  \Gamma \end{array}\right]> 0,
\end{equation}
\begin{equation}\label{kkw}
\Lambda_i|_{z=C}<0,\; \Lambda_i|_{z=-C}<0,\; i=1,2
\end{equation}
where
\begin{equation*}
\Lambda_1=\begin{pmat}[{|.}]
{ {\Theta_0}}&\Theta_1\cr\-
{{\ast} } &  -rh\cr   
\end{pmat},
\end{equation*}
\begin{equation*}
\Lambda_2=\begin{pmat}[{|.}]
{ {\Theta_2}}&{\Theta_3}\cr\-
{{\ast}}& {-rh}\cr
\end{pmat},
\end{equation*}
$\Theta_0=\{\theta_{ij}\}$ is a symmetric matrix composed from 
\begin{equation*}
\begin{array}{ll}
\theta_{11}=2p_1(1-\kappa)+\lambda_1-\lambda_2,\;
\theta_{22}=-2p_1\kappa+\lambda_1-\lambda_2,\\
\theta_{33}=-2p_1+2\delta p_2,\;
\theta_{35}=-p_2(1-\kappa),\;
\theta_{36}=-\dfrac{\lambda_2}{2},\\
\theta_{44}=-2p_1+2\delta p_2,\;
\theta_{45}=p_2\kappa,\;
\theta_{46}=-\dfrac{\lambda_2}{2},\\
\theta_{55}=-2p_2,\;
\theta_{56}=-p_2\mu,\;
\theta_{57}=p_2\mu,\\
\theta_{66}=-2p_1\mu+2\delta p_1-\dfrac{\pi^2}{2}\lambda_1,\;
\theta_{67}=p_1\mu,\\
\theta_{77}=-\eta,
\end{array}
\end{equation*}
\begin{equation}
\Theta_1=[\begin{array}{ccccccc}-r h z &0&-(1-\kappa)rh&\kappa rh& -rh&-\mu rh&\mu rh\end{array}]^T
\end{equation}
\begin{equation}
\Theta_2=\begin{pmat}[{......|}]
&&&&{ {\Theta_0}} & &&\ast \cr\-
0&0&0&0&p_2\mu h&p_1\mu h&0&-rhe^{-2\delta h}\cr
\end{pmat},
\end{equation}
\begin{equation}
\Theta_3=[\begin{array}{cc}\Theta_1&\mu rh^2\end{array}]^T
\end{equation}
Then for any initial state $z_0\in V$ satisfying $\|z_0\|_V^2<C^2$,  a unique solution of  \dref{YAZI} exists and  satisfies
$$\begin{array}{ll} p_1\|z\|^2_{L^2(\Omega)} +p_2\|\Delta z\|^2_{L^2(\Omega)}\le e^{-2\sigma t} [p_1\|z_0\|^2_{L^2(\Omega)} +p_2\|\Delta z_0\|^2_{L^2(\Omega)}],\;t\ge 0,
\end{array}$$
where $\sigma$ is  a unique positive  solution of \dref{hdou}.
\end{theorem}
\begin{proof}
See Appendix.	
\end{proof}

\section{Numerical example}
Consider  the system \dref{a} under the sampled-data control law \dref{y} with  the averaged measurements \dref{w}. Here we choose 
$\mu=0.95$.
By verifying LMI conditions of Theorem \ref{beibei} with $\delta=0.1$, $\kappa=-0.5$, 
 $\bar \Delta=1/4$, $C=2$, we find that
the closed-loop system \dref{niu} preserves the exponential stability within
a given domain of initial conditions $\|z_0\|_V^2<4$ for  $t_{k+1}-t_k \le h\le  0.39$. Note that $\bar  \Delta=1/4$ corresponds to $N=16$  square subdomains with the sides length $1/4$.
The feasible solutions of LMIs with $h=0.35$ are given as follows:
$p_1=80.6354$, $p_2=5.145$.

We compute the solution of the closed-loop system \dref{niu} numerically via finite element method.
Let $\xi=0.1$ and $M=1/\xi=10$. Define $(x_{1}^i,x_2^j)=(i\xi,j\xi)$, $i,j=0,1,2,\dots, M$.  
We divide $\Omega$ on $M^2=100$ squares $R_{ij}$ defined by
$$R_{ij}\triangleq \{(x_1,x_2)\in \Omega |x_{1}^i\le x_1\le x_{1}^{i+1}, x_{2}^j\le x_2\le x_{2}^{j+1}\}.$$
On the node $(x_{1}^i,x_2^j)$, four finite element basis functions are selected as 
\begin{equation*}
\begin{array}{ll}
N_1^{i,j}(x)=\left\{\begin{array}{ll}\!\!\left(1-\frac{x_1-x_1^i}{\xi}\right)\left(\frac{x_2-x_2^j}{\xi}\right),\; &x\in R_{ij},\\
0,\; &otherwise\end{array}\right.\\
N_2^{i,j}(x)=\left\{\begin{array}{ll}\!\!\left(1-\frac{x_1-x_1^i}{\xi}\right)\left(1-\frac{x_2-x_2^j}{\xi}\right),\; &x\in R_{ij},\\
0,\; &otherwise\end{array}\right.\\
N_3^{i,j}(x)=\left\{\begin{array}{ll}\!\!\left(\frac{x_1-x_1^i}{\xi}\right)\left(1-\frac{x_2-x_2^j}{\xi}\right),\; &x\in R_{ij},\\
0,\; &otherwise\end{array}\right.\\
N_4^{i,j}(x)=\left\{\begin{array}{ll}\!\!\left(\frac{x_1-x_1^i}{\xi}\right)\left(\frac{x_2-x_2^j}{\xi}\right),\; &x\in R_{ij},\\
0,\; &otherwise\end{array}\right.\\
\end{array}
\end{equation*}
We consider the Galerkin approximation solution of the closed-loop system in finite dimensional space
 generated  by these basis functions, which takes the form 
$$z^{M}(x,t)=\sum_{k=1}^4\sum_{i,j=1}^M m_k^{i,j}(t)N_k^{i,j}(x),$$
 where $m_k^{i,j}(t)$ are determined by standard finite element Galerkin method to satisfy some ODEs. 
  Fig. 2 shows  snapshots of the state $z(x,t)=z(x_1,x_2,t)$ at different times for the closed-loop system \dref{niu} with $t_{k+1}-t_k=0.35$, $N=16$ and initial condition  $z(x_1,x_2, 0)\! =\!0.236\! \sin (\pi x_1)\!\sin (\pi x_2)$,\; $(x_1,x_2)\!\in \Omega=(0,1)^2$.
 It is seen that the closed-loop system is  stable.
  Fig. 3 demonstrates the time evolution of $V_1(t)$ via the finite difference method, where the steps of space and time are taken
as $1/4$ and $0.00025$, respectively.

By verifying the LMI conditions of Theorem 1, we obtain the maximum value $h=0.39$  that preserves the exponential stability. By simulation of the solution to the closed-loop system starting from the same initial condition, we find that stability is preserved for essentially larger values of $h$ till approximately $h=2.45$. 
 \begin{figure}[h]
	\begin{center}
		\subfigure	
		{\includegraphics[width=0.7\linewidth]{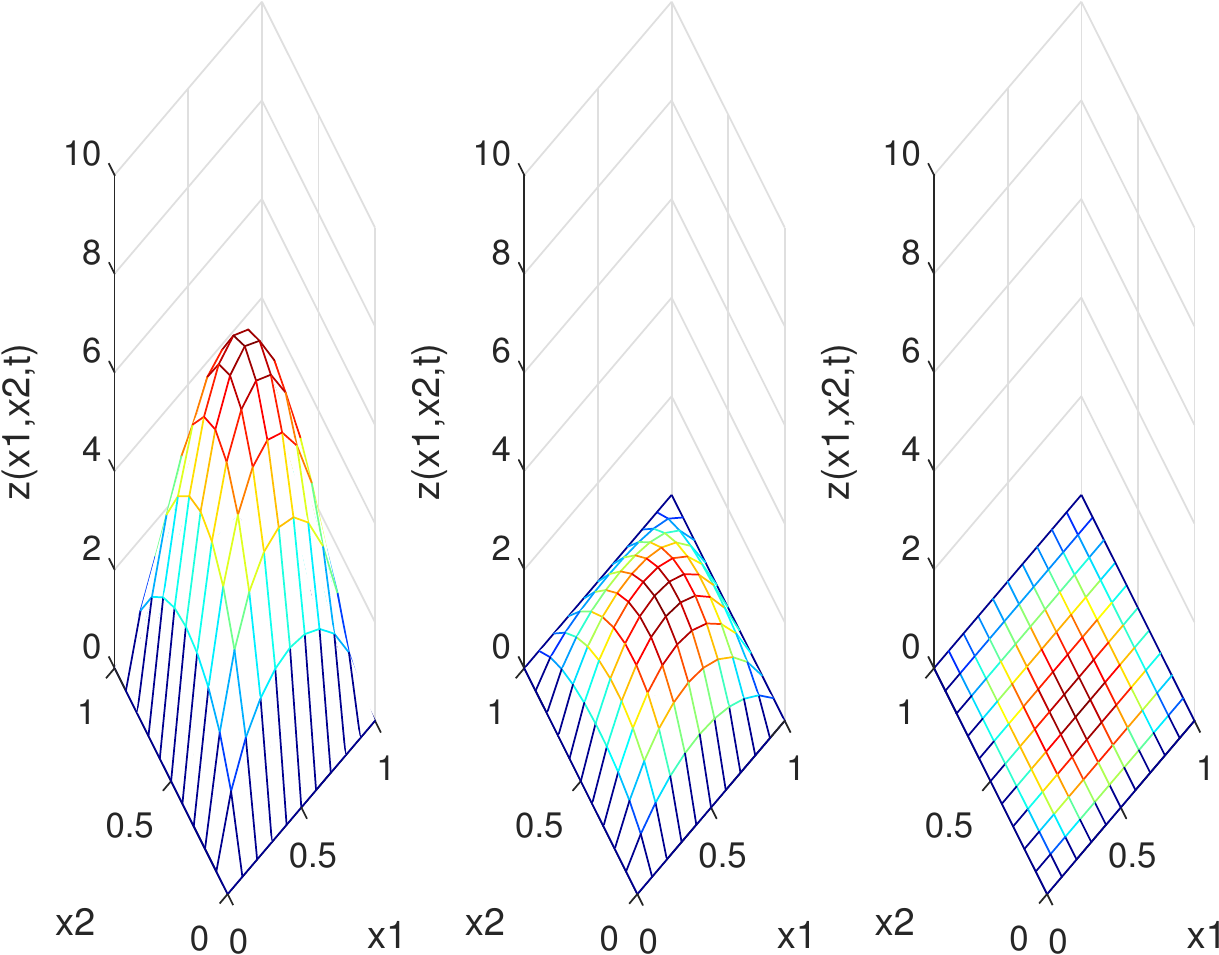}} 
		\caption{Snapshots of state $z(x_1,x_2,t)$ at different time $t\in \{0,1.4,14\}$}
	\end{center}
\end{figure}
\begin{figure}[h]
	\begin{center}
		\subfigure	
		{\includegraphics[width=0.7\linewidth]{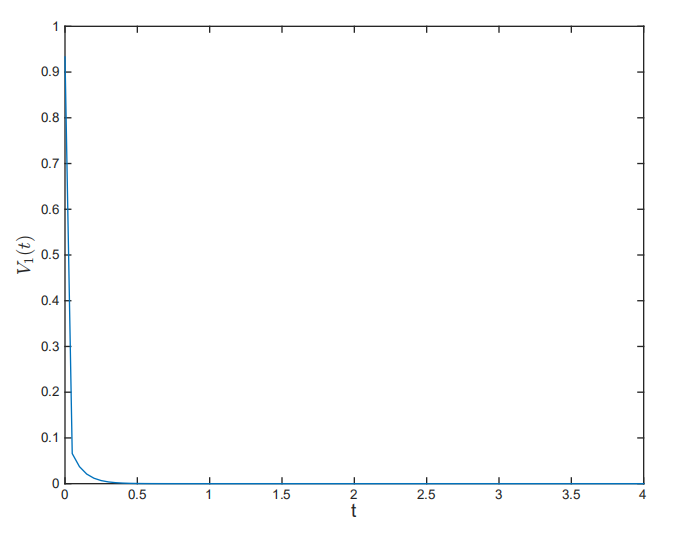}} 
		\caption{Lyapunov function $V_1(t)$}
	\end{center}
\end{figure}

For the sampled-data controller \dref{y}  under the point
measurement \dref{swift}, by choosing $\mu=0.95$,
 and using Yalmip we verify LMI conditions of Theorem \ref{woy} with   $\delta=0.2$, $\delta_1=0.15$, $\kappa=-0.5$, $\bar \Delta=1/4$, $C=2$. Then we find that the resulting closed-loop system is  exponentially stable  for  $t_{k+1}-t_k \le h\le  0.37$ for  any initial values satisfying $ \| z_0\|_{L^2(0,1)} < 1$.  
 
 Since point measurements use less information
 on the state, the point measurements allow smaller
 sampling intervals than averaged measurements. 
Simulations of the solutions to the closed-loop system under the point measurements confirm the theoretical results.
\section{Conclusion}
The present paper discusses sampled-data control of 2D KSE 
under the spatially distributed  averaged or point measurements. Sufficient LMI conditions have been investigated such that the regional stability of the closed-loop system is guaranteed. 

Our results are applicable to   sampled-data controller design of high dimensional distributed parameter systems. Our next step places its main focus on
  $H_\infty$ filtering problem of high dimensional coupled ODE-PDE/PDE-PDE system.

\appendix
\renewcommand\thesection{\appendixname~\Alph{section}}
\renewcommand\theequation{\Alph{section}.\arabic{equation}}
\renewcommand\thelemma{\Alph{section}.\arabic{lemma}}
\renewcommand\thetheorem{\Alph{section}.\arabic{theorem}}
	\begin{center}
	Proof of Theorem \ref{woy}
	\end{center}
Step 1: We have shown that there exists a unique local strong solution  to \dref{YAZI}  on some interval 
$[0,T]\subset [0,t_1]$. By Theorem 6.23.5 of \cite{kran}, we obtain that the solution exists on the entire interval $[0,t_1]$ if this solution admits a priori estimate. In Step 3, it will be shown that the feasibility of LMIs \dref{yangyang}, \dref{ninian} guarantees that the solution  of \dref{YAZI} admits  a priori bound, which can further guarantees the existence of the solution for all $t\geq 0$.

Step 2: Assume formally that there exists a  strong solution of \dref{YAZI} starting from $\|z_0\|_V<C$ for all $t\geq 0$.
	Differentiating $V_1$ along \dref{YAZI}, we obtain the inequality \dref{yanzi}. 	
	Denote
		\begin{equation}
	f_j(x,t)=z(x,t)-z(\bar x_j,t)=\int_{\bar x_j}^x z_\xi(\xi,t)d\xi,
	\end{equation}
	From Lemma \ref{anton}, we have
	\begin{equation}
	\begin{array}{ll}
	\eta\|f_j\|_{L^2(\Omega_j)}^2\!\!& \le\! \beta_1\left(\!\frac{\bar \Delta}{\pi}\right)^2\!\!\!\|z_{ x_1}\|_{L^2(\Omega_j)}^2\! +\beta_2\left(\!\frac{\bar \Delta}{\pi}\right)^2\!\!\!\|z_{ x_2}\|_{L^2(\Omega_j)}^2\\&
	+\beta_3\left(\frac{\bar \Delta}{\pi}\right)^4\!\!\|z_{x_1 x_2}\|_{L^2(\Omega_j)}^2
	\end{array}
	\end{equation}
	for any  scalars $\beta_1$, $\beta_2$, $\beta_3$ such that \dref{mingbai} holds.\\
	Hence, 
	\begin{equation}\label{any}
	\begin{array}{ll}
	\sum\limits_{j=1}^{N}&\left[\beta_1\left(\frac{\bar \Delta}{\pi}\right)^2\|z_{ x_1}\|_{L^2(\Omega_j)}^2 +\beta_2\left(\frac{\bar \Delta}{\pi}\right)^2\!\|z_{ x_2}\|_{L^2(\Omega_j)}^2\right.\\&\left.
	+\beta_3\left(\frac{\bar \Delta}{\pi}\right)^4\|z_{x_1 x_2}\|_{L^2(\Omega_j)}^2-\eta\|f_j\|_{L^2(\Omega_j)}^2\right]\geq 0.
	\end{array}
	\end{equation}
	Denote
	\begin{equation*}
	\rho(x,t)\triangleq\frac{1}{t-t_k}\int_{t_k}^t z_s(x,s)ds.
	\end{equation*}
	Then, we have 
	\begin{equation*}
	z(x,t)=z(x,t_k)+(t-t_k)\rho(x,t).
	\end{equation*}
By using Jensen's inequality, we have
\begin{equation}\label{mimi}
\begin{array}{ll}
-r \int_{ \Omega}\int_{t_k}^t e^{2\delta (s-t)}z_s^2(x,s)ds dx
\le -r e^{-2\delta h}(t-t_k) \int_{ \Omega} \rho ^2(x,t)dx.
\end{array}
\end{equation}
	Integration by parts leads to 
\begin{equation}\label{gigi}
\begin{array}{ll}
2p_1\int_{ \Omega}zz_tdx
=2p_1(1-\kappa)\int_\Omega z_{x_1}^2dx-2p_1\kappa \int_\Omega z_{x_2}^2dx-2p_1\int_\Omega |\Delta z|^2dx\\
+2p_1\mu \sum\limits_{j=1}^N \int_{\Omega_{j} }z\left[(t-t_k)\rho(x,t)+f_j(x,t_k)\right]dx\\
-2p_1\mu \int_\Omega z^2dx,
\end{array}	
\end{equation}
\begin{equation}\label{kiki}
\begin{array}{ll}
2p_2\int_\Omega \Delta z \Delta z_tdx
=2p_2\int_\Omega \Delta^2 z \cdot  z_tdx\\
=2p_2\int_\Omega \Delta ^2 z [-zz_{x_1} -(1-\kappa)z_{x_1x_1}+\kappa z_{x_2x_2}-\Delta^2 z]dx\\
-2p_2\mu \sum\limits_{j=1}^N\int_{\Omega_j}\Delta ^2 z [z(x,t)-(t-t_k)\rho(x,t)-f_j(x,t_k)]dx.
\end{array}
\end{equation}
Then from \dref{mimi}-\dref{kiki},  adding \dref{any} into $\dot V_1+2\delta V_1$ we obtain
\begin{equation}\label{juan}
\begin{array}{ll}
&\dot V_1(t)+2\delta V_1(t)-\delta_1\sup\limits_{\theta\in [-h,0]}V_1(t+\theta)
\le \dot V_1(t)+2\delta V_1(t)-\delta_1 V_1(t_k)\\
&+\lambda_1\left [\int_{\Omega} \nabla z^T \nabla z dx -\frac{\pi^2}{2}\int_\Omega z^2 dx\right]\\
&+\lambda_2\left[-\int_{\Omega}\nabla z^T \nabla z dx-\int_{\Omega}z\Delta zdx\right]\\
&+\left[\beta_1\left(\!\frac{\bar \Delta}{\pi}\right)^2\int_{\Omega}z_{ x_1}^2(x,t_k)dx +\beta_2\left(\!\frac{\bar \Delta}{\pi}\right)^2\int_{ \Omega}z_{ x_2}^2(x,t_k)dx\right.\\
&\left.+\beta_3\left(\frac{\bar \Delta}{\pi}\right)^4\int_{\Omega}z_{x_1 x_2}^2(x,t_k)dx-\eta \sum\limits_{j=1}^N\int_{\Omega_j}f_j^2(x,t_k)dx\right]\\
&\le 2p_1(1-\kappa)\int_\Omega z_{x_1}^2dx-2p_1\kappa \int_\Omega z_{x_2}^2dx-2p_1\int_\Omega |\Delta z|^2dx\\
&+2p_1\mu \sum\limits_{j=1}^N \int_{\Omega_{j} } z\left[(t-t_k)\rho+f_j(x,t_k)\right]dx-2p_1\mu \int_\Omega z^2dx\\
&+2p_2\int_\Omega \Delta ^2 z [-zz_{x_1} -(1-\kappa)z_{x_1x_1}+\kappa z_{x_2x_2}-\Delta^2 z]dx\\
&-2p_2\mu \int_\Omega\Delta ^2 z [z-(t-t_k)\rho-f_j(x,t_k)d\xi]dx\\
&-r e^{2\delta h}(t-t_k) \int_{ \Omega} \rho ^2(x,t)dx
+r(t_{k+1}-t)\int_{\Omega}z_t^2(x,t)dx\\
&+2\delta p_1 \int_\Omega z^2dx+2\delta p_2 \int_\Omega |\Delta z|^2dx
-\delta_1 p_1 \int_\Omega z^2(x,t_k)dx\\
&-\delta_1 p_2 \int_\Omega |\Delta z(x,t_k)|^2dx+(\lambda_1-\lambda_2)\int_{ \Omega}[z_{x_1}^2+z_{x_2}^2]dx\\
&-\lambda_1\frac{\pi^2}{2}\int_\Omega z^2dx-\lambda_2\int_\Omega z\Delta zdx+\beta_1\left(\frac{\bar \Delta}{\pi}\right)^2\int_{\Omega}z_{ x_1}^2(x,t_k)dx \\&+\beta_2\left(\frac{\bar \Delta}{\pi}\right)^2\int_{ \Omega}z_{ x_2}^2(x,t_k)dx
+\beta_3\left(\frac{\bar \Delta}{\pi}\right)^4\int_{\Omega}z_{x_1 x_2}^2(x,t_k)dx
-\eta \sum\limits_{j=1}^N\int_{\Omega_j}f_j^2(x,t_k)dx.
\end{array}
\end{equation}
Using \dref{jin2}, we have 
\begin{equation}\label{sun}
\begin{array}{ll}
&-2p_1\int_{ \Omega}|\Delta z|^2 dx\\
&=-2p_1 \int_{\Omega} [z_{x_1x_1}^2(x,t)+z_{x_2x_2}^2(x,t)+2z_{x_1x_2}^2(x,t)]dx,
\end{array}
\end{equation}
\begin{equation}\label{anzui}
\begin{array}{ll}
&-\delta_1p_2  \int_{\Omega}|\Delta z(x,t_k)|^2dx\\
&=-\delta_1p_2 \int_{\Omega} [z_{x_1x_1}^2(x,t_k)+z_{x_2x_2}^2(x,t_k)+2z_{x_1x_2}^2(x,t_k)]dx.
\end{array}
\end{equation}
Set 
\begin{equation*}
\begin{array}{ll}
\eta_0=&\{z_{x_1}(x,t),z_{x_2}(x,t),z_{x_1x_1}(x,t),z_{x_2x_2}(x,t),\Delta^2 z(x,t),\\&z(x,t),f_j(x,t_k)\},\\
\eta_1=&\{z_{x_1}(x,t),z_{x_2}(x,t),z_{x_1x_1}(x,t),z_{x_2x_2}(x,t),\Delta^2 z(x,t),\\&z(x,t),f_j(x,t_k),\rho(x,t)\},\\
\bar \eta=&\{z(x,t_k),z_{x_1x_1}(x,t_k),z_{x_2x_2}(x,t_k)\}
\end{array}
\end{equation*}
Substituting \dref{sun}, \dref{anzui} into \dref{juan}, we obtain
 \begin{equation}\label{shu}
 \begin{array}{ll}
 &\dot V_1(t)+2\delta V_1(t)-\delta_1\sup\limits_{\theta\in [-h,0]}V_1(t+\theta)\\
 &\le \sum\limits_{j=1}^N\disp \int_{\Omega_j} \dfrac{h-t+t_k}{h}\eta_0^T\Theta_0\eta_0 +\dfrac{t-t_k}{h}
\eta_1^T\Theta_1\eta_1 dx\\
&+\int_{ \Omega}\bar\eta^T\bar \Theta \bar \eta dx
+r(t_{k+1}-t)\int_{ \Omega}z_t^2(x,t)dx\\
&-(4p_1-4\delta p_2)\int_{ \Omega}z_{x_1x_2}^2(x,t)dx\\
&-[2\delta_1p_2-\beta_3(\frac{\bar\Delta}{\pi})^4]\int_{\Omega}z_{x_1 x_2}^2(x,t_k)dx.
 \end{array}
 \end{equation}
 Substitution of $z_t$ from \dref{YAZI} yields
 \begin{equation}\label{airen}
 \begin{array}{ll}
 r(t_{k+1}-t)\int_{\Omega}z_t^2(x,t)dx\\
 \le rh\sum\limits_{j=1}^N\int_{\Omega_j} [-zz_{x_1}-\!(1-\kappa) z_{x_1x_1}\!+\kappa z_{x_2x_2}- \Delta^2 z\\
 -\mu z+ \mu (t-t_k)\rho+\mu f_j(x,t_k)]^2dx.
 \end{array}
 \end{equation}
 Set $\bar\psi={\rm col}\{z_{x_1}(x,t), z_{x_1x_1}(x,t), z_{x_2x_2}(x,t), \Delta^2 z(x,t), \\
 z(x,t), f_j(x,t_k), \rho(x,t)\}$ and
 \begin{equation*}
 \beta \triangleq\left[\begin{array}{cccccccc} -z& -(1-\kappa) & \kappa &-1&-\mu&\mu & \mu(t-t_k)\end{array}\right].
 \end{equation*}
 Then
 \begin{equation}\label{mami}
 \begin{array}{ll}
 [-zz_{x_1}-\!(1-\kappa) z_{x_1x_1}\!+\kappa z_{x_2x_2}- \Delta^2 z
 -\mu z+\mu f_j(x,t_k)\\+ \mu (t-t_k)\rho]^2
 = \bar\psi^T\beta^T\beta\bar \psi.
 \end{array}
 \end{equation}
 Application of Schur complement theorem to \dref{mami}, together with \dref{shu} and \dref{airen}, implies
 \begin{equation*}
 \dot V_1(t)+2\delta V_1(t)-\delta_1\sup\limits_{\theta\in [-h,0]}V_1(t+\theta)\le 0
 \end{equation*}
 if \dref{jinhan}-\dref{hans} are satisfied, and $\Lambda_1<0$, $\Lambda_2<0$ hold for all $z\in (-C,C)$.
 Similar to Theorem \ref{beibei}, LMI \dref{kkw} imply $\Lambda_1<0$, $\Lambda_2<0$ for all $z\in (-C,C)$. Thus the result is established via Halanay's inequality.
 
 Step 3: The feasibility of LMIs \dref{jinhan}-\dref{kkw} yields that the solution of \dref{YAZI}  admits a priori estimate 
$V_1(t)\le e^{-2\sigma t} \sup\limits_{-h\le \theta \le 0} V_1(\theta)$. By Theorem 6.23.5 of \cite{kran}, continuation of this solution under a priori bound to entire interval $[0,\infty)$.


\begin{thebibliography}{99}
			
				
				
	\bibitem{kuramoto}Kuramoto, Y., Tsuzuki, T. (1975). On the formation of dissipative structures in reaction-diffusion systems.  {\it Progr. Theoret.Phys.}, 54, pp. 687-699.
	
	\bibitem{bao} Feng, B.F., Malomed, B.A., Kawahara, T. (2003). Cylindrical solitary pulses in a two-dimensional stabilized 	Kuramoto-Sivashinsky system. {\it Physica D.}, 175, 127-138.
							\bibitem{Armaou2} Armaou, A., Christofides, P.D. (2000). Wave suppression by nonlinear finite-dimensional control. {\it Chem.Eng. Sci.}, 55, pp. 2627-2640.
						
					\bibitem{Armaou1} Armaou, A., Christofides, P.D. (2000). Feedback control of the Kuramoto-Sivashinsky Equation. {\it Physica D}, 137, pp. 49-61.
						\bibitem{Armaou} Christofides, P.D., Armaou, A. (2000). Global stabilization of the Kuramoto-Sivashinsky equation via distributed output feedback control. {\it Systems \& Control Letters}, 39, pp. 283-294.
						
				\bibitem{Titi}Lunasin, E., Titi, E.S. (2017). Finite determining parameters feedback control for distributed nonlinear dissipative systems - a computational study.  {\it Evolution Equations and Control Theory}, 6, pp. 535-557.	
				
				
					\bibitem{Krstic}Liu, W.-J., Krstic, M. (2001). Stability enhancement by boundary control in the Kuramoto-Sivashinsky equation. {\it Nonlinear Anal. Ser. A: Theory Methods},
					43, pp. 485-507.
					
						\bibitem{LU}Coron, J.-M., L\"{u}, Q. (2015). Fredholm transform and local rapid stabilization for a Kuramoto-Sivashinsky equation. {\it J. Differential Equations}, 259, pp. 3683-3729.

 \bibitem{kang2} Kang, W., Fridman, E. (2019).
			 Distributed stabilization of Korteweg-de Vries-Burgers equation in
                         the presence of input delay. {\it Automatica}, 100, 260-273. 						
	
				
  \bibitem{Emilia3} Fridman, E., Bar Am, N. (2013). Sampled-Data Distributed $H_{\infty}$ Control of Transport Reaction Systems. {\it SIAM Journal on Control and Optimization}, 51, pp. 1500-1527.
               
  \bibitem{Emilia2} Fridman, E., Blighovsky, A. (2012). Robust sampled-data control of a class of semilinear parabolic systems. {\it  Automatica}, 48, pp. 826-836.
               
 
      \bibitem{kang} Kang, W., Fridman, E. (2018).
			Distributed sampled-data control of Kuramoto-Sivashinsky equation.	  {\it Automatica}, 95, 514-524. 
			 

                         
	\bibitem{Emilia1}	 Bar Am., N.,  Fridman, E. (2014). Network-based Distributed $H_{\infty}$-Filtering of Parabolic Systems. {\it Automatica}, 50, pp. 3139–-3146.
                         
	\bibitem{Anton1}Selivanov, A., Fridman, E. (2019). Delayed $H_\infty$ control of 2D diffusion systems under delayed pointlike measurements.  {\it Automatica}, 109, 108541.

 \bibitem{jones} Jones, D., Titi, E.S. (1993) Upper bounds on the number of determining modes, nodes, and volume elements for the Navier-Stokes equations. {\it Indiana University Mathematics Journal}, 42, pp. 875. 
                         
	\bibitem{Ruben}Tomlin, R., Gomes, S.N. (2019). Point-actuated feedback control of multidimensional interfaces. 
	{\it IMA Journal of Applied Mathematics}.
				
	\bibitem{Titi2} Azouani, A., Titi, E.S. (2014). Feedback control of nonlinear dissipative systems by finite determining parameters - A reaction-diffusion Paradigm. {\it Evolution Equations and Control Theory}, 3, pp. 579-594.
	
		\bibitem{Payne} Payne, L. and Weinberger, H. (1960).  An optimal Poincar\'e inequality for
			convex domains, {\it Archive for Rational Mechanics and Analysis}, 5, pp. 286-292.
						
	\bibitem{Hardy}	Hardy, G. H., Littlewood, J. E.,  Polya, G. (1988). Inequalities. Cambridge: Cambridge
University Press.
							
	\bibitem{Anton}  Selivanov, A., Fridman, E. (2017). Sampled-data relay control of diffusion PDEs. Automatica, 82, pp. 59-68.
	\bibitem{Bellman}Bellman, R.E., Cooke, K.L. (1963). {\it  Differential-difference equations}. RAND Corporation.
%
\bibitem{halanay}	Halanay, A. (1966). Differential equations: Stability, oscillations, time lags. New York:
	Academic Press.						

%
	\bibitem{Emilia4} Fridman, E. (2014). {\it Introduction to Time-Delay Systems: Analysis and Control}. Basel: Birkh\"{a}user.
%
%
%
%
%
%
%
%
%
%
%
%
				\bibitem{Henry}	Henry, D. (1981). {\it  Geometric theory of semilinear parabolic equations}. New York:
				Springer-Verlag.
%
%
%
%
%
%
				\bibitem{kran}	 Krasnoselskii, M. A.,  Zabreiko, P. P.,  Pustylii, E. L., Sobolevskii, P. E. (1976).
				{\it Integral operators in spaces of summable functions}. Springer Netherlands.
%
%
%
%
%
%
%
%
%
%
%
%
%
%
%
%
				\bibitem{Pazy} Pazy, A. (1983). {\it Semigroups of linear operators and applications to partial differential equations}. New York: Springer-Verlag.
%
%
%
%
%
%
\bibitem{CDC19}Kang, W., Fridman, E. (2019). Sampled-data control of 2D Kuramoto-Sivashinsky equation under the averaged measurements. {\it Proceeding of the 58th IEEE Conference on Decision and Control (CDC)}.
	\bibitem{ECC18}Kang, W., Fridman, E. (2018). Sampled-data control of Kuramoto-Sivashinsky equation under the point measurements. {\it Proceeding of the 16th European Control Conference  (ECC)}.
\bibitem{Yakubovich}Yakubovich, V.  A.(1971) {\it S-procedure in nonlinear control theory.} Vestn Leningr Univ 1: 62-77.
\bibitem{James} Robinson, J.C. (2001). {\it Infinite-Dimensional Dynamical Systems: An introduction to dissipative parabolic PDEs and the theory of global attractors}.
			Cambridge University Press.
\bibitem{pingping} Rektorys, K. (2012). {\it Variational methods in mathematics, science and engineering}. Springer Science \& Business Media.			
\bibitem{krstic3}	 Krstic, M., Smyshlyaev, A. (2008). {\it Boundary control of PDEs: A course on backstepping designs.} Philadelphia, PA: SIAM.	
\bibitem{emilia5} Fridman, E. (2010). A refined input delay approach to sampled-data control. {\it Automatica}, 46, pp.421-427. 

\bibitem{zyl} Liu, Z., Zheng, S. (1999). Semigroups associated with dissipative systems. {\it Res. Notes Math.}, 398, Chapman \& Hall/CRC, Boca Raton.

			\end{thebibliography}
     \end{document}